\documentclass[12pt]{article}
\usepackage[margin=1.1in]{geometry}

\usepackage{authblk}
\usepackage{amsthm}
\usepackage{amssymb}
\usepackage{amsmath}
\usepackage{float}
\usepackage{tikz}
\usetikzlibrary{arrows,calc,shapes,decorations.pathreplacing}
\usepackage{hyperref}
\hypersetup{pdftex,colorlinks=true,allcolors=blue}
\usepackage{caption}
\usepackage{pgfplots}
\usepackage{subcaption}
\usepackage{enumitem}
\usepackage{algorithm}
\usepackage{algorithmicx}
\usepackage{algpseudocode}

\newenvironment{customlegend}[1][]{%
    \begingroup
    \csname pgfplots@init@cleared@structures\endcsname
    \pgfplotsset{#1}%
}{%
    \csname pgfplots@createlegend\endcsname
    \endgroup
}%
\def\addlegendimage{\csname pgfplots@addlegendimage\endcsname}

\newtheorem{theorem}{Theorem}
\newtheorem{lemma}[theorem]{Lemma}
\newtheorem{proposition}[theorem]{Proposition}

\newtheorem*{theorem*}{Theorem}
\newtheorem*{definition}{Definition}

\theoremstyle{remark}

\newtheorem{remark}{Remark}
\newtheorem*{remark*}{Remark}

\def\P{\mathop{\mathrm{P}}}
\def\E{\mathrm{E}}

\DeclareMathOperator*{\argmin}{arg\,min}

\begin{document}

\title{Delay-minimizing capacity allocation in an infinite server queueing system\footnote{To appear in Stochastic Systems.}
}
\author{Refael Hassin and Liron Ravner\footnote{lravner@post.tau.ac.il}}
\affil{Department of Statistics and Operations Research \\ Tel Aviv University}

\date{\today}
\maketitle

\begin{abstract}
We consider a service system with an infinite number of exponential servers sharing a finite service capacity. The servers are ordered according to their speed, and arriving customers join the fastest idle server. A capacity allocation is an infinite sequence of service rates. We study the probabilistic properties of this system by considering overflows from sub-systems with a finite number of servers. Several stability measures are suggested and analysed. The tail of the series of service rates that minimizes the average expected delay (service time) is shown to be approximately geometrically decreasing. We use this property in order to approximate the optimal allocation of service rates by constructing an appropriate dynamic program. 
\end{abstract}

\section{Introduction}\label{sec:Intro}
We are interested in the optimal allocation of service rates in a system with an infinite number of parallel servers and finite service capacity. Customers join the fastest server available, without jockeying if a faster server becomes available at a later time. This model is appropriate for a system with servers at different physical locations and no possibility to accommodate waiting customers. For example this may be the case in a distributed (cloud) computing system with jobs arriving at a central router that immediately sends them to the best available server. Other applications with ordered entry and heterogeneous servers are conveyor and storage systems. In this setting an allocation is given by an infinite sequence of service rates. Every such service-rate sequence determines the probabilistic traits of the system. Our objective is to minimize the stationary expected delay (service time) faced by arrivals.

From a practical point of view the infinite-server setting is of course aimed to be a good approximation of large-scale systems. It is important to note that in this respect the model presented here does indeed capture the behaviour of such systems when the capacity allocation is a ``sensible'' one. In particular, for service-rate sequences that satisfy certain stability and delay conditions that will be defined in the next sections, the blocking probability from finite sub-systems goes to zero very fast with the number of servers. This means that for a even a moderately sized system the probability of an arrival facing a full system is negligible. We show that the blocking probability decreases exponentially with the number of servers. Moreover, we provide a framework that enables the approximation of the blocking probability for a large number of servers, and thus for the required number of servers for the probability to be smaller than a given threshold.

The model is an ordered GI/M/$\infty$ system: independent and identically distributed inter-arrival times, exponential service times with heterogeneous rates and customers routed to the fastest idle server upon their arrival. An in-depth analysis of an M/M/$\infty$ with identical servers that are ordered (geographically), including heavy traffic approximations, can be found in \cite{book_N1984}. The assumption that servers are identical implies that the service capacity is infinite and the focus in \cite{book_N1984} is on distributional properties such as how many of the first $n$ servers are busy. The finite-capacity system studied in this paper is therefore very different and the analysis relies on the probability of blocking and overflows from finite server sub-systems. In particular, our methods rely on \cite{T1959} where the overflow distribution in a homogenous multi-server system with non-Markovian arrivals is characterized, and the extensions of \cite{TN1975} and \cite{Y1987} that account for heterogeneous servers. A key feature for our analysis is that blocking probabilities can be written as a product of Laplace-Stieltjes Transforms corresponding to the overflow times from subsets of the system. We leverage this structure to derive the expected delay in our infinite server system. Even though there is no queue, the number of busy servers is potentially unbounded while the output rate of the system is bounded by the finite capacity, unlike in typical infinite server settings where the output rate grows with the number of busy servers. Consequently, an important issue that arises in the infinite server model with finite service capacity is that the stability of the system depends on the allocation of service rates, and it is not enough to assume that the external arrival rate is lower than the total capacity (i.e.\ $\rho<1$). This is because the system is not work conserving in the sense that fast servers may be idle while customers are being served by slower servers. Intuitively, the service allocation needs to balance between fast rates at the good (fast) servers while still leaving enough capacity to handle overflows to the bad (slow) servers. These issues are addressed in detail in Section \ref{sec:stability}, where conditions for stability and finite expected delay are established. In particular we show that the service-rate sequence cannot decay faster than a geometric series with a decay parameter that is determined by the overflow probabilities.

The trade-off between the rate of capacity assignment, i.e.\ how much of the remaining capacity is assigned to a server when sequentially allocating from the fastest to the slowest, and the overflow probabilities is also at the core of the delay-minimization problem. While the input of the problem is quite simple: the inter-arrival distribution (a single parameter if the arrival process is Poisson), the decision variable is an infinite sequence. This leads to analytical as well as computational challenges. We formulate the optimal capacity allocation problem as an infinite dynamic program. However, the dynamic program is intractable because of the elaborate state and actions spaces. To this end we derive an asymptotically optimal geometric tail of the service-rate sequence. The asymptotic optimal geometric rate is shown to be the square-root of the term in the product representing the aforementioned overflow (blocking) probabilities. Furthermore, we use the geometric approximation of the tail in order to define a finite dynamic program which can be solved efficiently. Numerical analysis suggests that the optimal service-rate sequence is very close to geometric from the start. This means that instead of solving the original capacity allocation problem we can approximate the solution by the single parameter problem of finding the optimal geometric service-rate sequence. Moreover, in the special case of a Poisson arrival process the simple heuristic of choosing a geometric service-rate sequence with decay rate $\sqrt{\rho}$ is quite close to the approximate optimal solution.

The use of stochastic queueing models in order to model cloud computing, also known as distributed or parallel computing, is very common (e.g. \cite{KMM2012} and \cite{VSTMR2014}). For example, the highly cited paper of \cite{KMM2012} uses a M/G/$m/m+r$ queueing system to approximate the performance. Our model is related to theirs in the case of $r=0$. They state that a common assumption in cloud computing models is that there is some positive blocking probability with a predefined upper bound constraint, and our model can be conveniently used to approximate such a system. The afformentioned papers, and many others that use stochastic queueing models for cloud computing, assume homogeneous servers. However, most cloud computing systems do in fact have heterogeneous servers (see, for example, \cite{BDD1998}, \cite{ZBS2003} and \cite{KMM2013}). Furthermore, the \textit{Fastest Server First} is a common policy implemented in such systems (see \cite{ZBS2003}). In this paper we present a framework that allows for the analysis of constrained capacity allocation in large scale heterogeneous server systems. Another useful aspect of our model is that it provides tools to study the trade-off between blocking probabilities and expected delay in finite server systems. Other relevant applications of our model are large scale conveyor and storage systems (see \cite{Y1987} for a discussion).

The research of ordered service systems with heterogeneous servers has mostly focused on analysing the blocking probability in loss systems, and their minimization in particular. In \cite{TN1975} it was shown that the optimal allocation of service rates in terms of minimizing blocking in an ordered Markovian system is heterogeneous. In \cite{NE1981} it was shown that the optimal sequence of service rates, in terms of minimizing blocking probabilities, is decreasing. Analysis of an ordered system with a general arrival process, along with the comparison methods for different entry order regimes, can be found in \cite{Y1987} and \cite{SY1987}. An interesting observation made in \cite{CP1989} is that the policy of \textit{Fastest Server First} is not necessarily optimal, for example when the slower server has a lower variance of service time. Optimal assignment to an ordered system with heterogeneous customer types that can only be served by some of the servers was studied in \cite{Ross2014}. The work presented here is related to \cite{HSY2015} which analysed the capacity allocation and pricing in a loss system possibly with heterogeneous servers. The objective function considered in that paper is different from the others because the objective is maximizing profit and not minimizing blocking probabilities. This objective required analysis of the expected waiting times, which will also be important in the analysis presented here. In \cite{RM1990} routing policies were analysed with the goal of minimizing holding costs. Approximation analysis of ordered homogeneous-server systems can be found in \cite{C1985} and \cite{K2004}, and heteregenous servers with a single queue (including a waiting buffer) under the \textit{Fastest Server First} policy appeared in \cite{A2005}. Another related work is \cite{AGS2013} that considered a service capacity allocation problem for a system of parallel queues and heterogeneous customer types using a heavy-traffic approximation.

\paragraph{Paper outline:}
In Section \ref{sec:model} we present the model and mention some of its known properties. We define system stability along with necessary conditions for finite expected delay in Section \ref{sec:stability}. In Section \ref{sec:geometric} we introduce the special class of service-rate sequence that decreases geometrically. We prove that the tail of the optimal service-rate sequence is of this type. This fact is due to the product form of the blocking probabilities. In Section \ref{sec:opt} we formally define our optimization problem as an infinite horizon dynamic program and suggest a numerical method to approximate its solution using the fact that the tail of the optimal sequence is geometric. We then proceed to present numerical analysis and examples of the optimal service-rate sequence in Section \ref{sec:numerical}. The numerical results suggest that the optimal service-rate sequence is very close to geometric from the start, and not just at the tail. Finally, Section \ref{sec:discussion} features concluding remarks and a brief discussion of straightforward extensions of our analysis aimed at optimizing other performance measures of the system, apart from expected delay.

\section{Model}\label{sec:model}
Customers arrive at a service system according to a renewal process with mean inter-arrival time $\E T_0=\frac{1}{\lambda}$. The system is comprised of an infinite number of parallel exponential servers that are ordered according to service-rate; $\mu_1\geq \mu_2\geq\cdots$, such that $\sum_{n=1}^\infty \mu_n=\mu>\lambda$. Every arriving customer joins the fastest server available, and does not switch server even if a faster server later becomes available while he is still in the system. For a given $\mu$, our goal in this paper is to find a sequence of service rates that minimizes the stationary expected sojourn time (delay) of an arriving customer.

Let $\mathbf{X}=(X_1,X_2,X_3,\ldots)$ be the random sequence of server indicators ,zero if idle and one if busy, at arrival times in the limit. The state space of the process can be defined as follows\footnote{This state space description was suggested by Brian Fralix.},
\[
\mathbf{X}\in\mathcal{S}=\bigcup_{n=1}^\infty A_n,
\]
where $A_n=\{x\in\{0,1\}^\infty:\ n=\sup\{j:\ x_j=1\}\}$. In words, $A_n$ is the set of all states such that the highest indexed  busy server is $n$. Note that $\mathcal{S}$ is a countable collection of finite sets and is therefore countable. The underlying continuous time process is not Markovian, due to the general arrival distribution, however the embedded process at arrival moments is indeed a discrete-time Markov chain. The state space should not be confused with the uncountable set $\{0,1\}^\infty$ that includes states with infinitely many ones. This space is ``too big" as the probability of the process being in a state $s\in\{0,1\}^\infty$ with an infinite number of ones is zero for any finite time, much like the queue length process of a single server queue that is defined on the set positive integers $\mathbb{Z}^+$ and not $\mathbb{Z}^+\cup \{\infty\}$.

\begin{remark*}
The random variables and distributional properties discussed in this work are all with respect to the limiting distribution $\mathbf{X}$ at arrival times, which is also the stationary distribution if the process is ergodic. This distribution may be different from the limiting time average distribution and the analysis does not require PASTA. Furthermore, all random variables depend on the service-rate sequence $\{\mu_n\}_{n=1}^\infty$, but we omit this from the notations for the sake of brevity.
\end{remark*}

Let $Y:=\inf\{i:X_i=0\}$ denote the random variable of the fastest available server, according to the limiting distribution at arrival times. Further denote by $S$ the respective limiting delay (service time) faced by an arbitrary arriving customer. The expected delay, is the expected service time at the fastest idle server upon arrival,
\[
\E S= \sum_{n=1}^\infty \frac{1}{\mu_n}\P(Y=n).
\]
We will soon argue that this limit is well defined even when there is no stationary distribution for the underlying process, in which case $\E S$ is infinite. To be specific, each probability $\P(Y=n)$ is derived from the limiting distribution of a Markov chain with a finite state space and therefore the infinite sequence is well defined and so is the sum.

The state of any server $n\geq 1$ depends only on the arrival process and on the service process at servers $i\leq n$. For example, if the arrival process is Poisson then by viewing the first server as an isolated M/M/$1/0$ system,
\[
\P(Y=1) = \P(X_1=0)=\frac{\mu_1}{\lambda+\mu_1}.
\]
The distribution of $Y$ is obtained from the blocking probabilities of consecutive sub-systems,
\begin{equation}\label{eq:Y_blocking}
\P(Y=n) = \P(X_i=1 \ \forall i\leq n-1)-\P(X_i=1 \ \forall i\leq n),\ n\geq 2,
\end{equation}
where $\P(X_i=1 \ \forall i\leq n)$ is the blocking probability in a GI/M$/n/n$ system with heterogeneous ordered servers. In the following analysis we use the more compact notation:
\begin{eqnarray*}
p_n &:=& \P(X_i=1 \ \forall i\leq n), \quad n\geq 1, \\
q_n &:=& \P(Y=n), \quad n\geq 1. 
\end{eqnarray*}

%
%
%

We are interested in computing \eqref{eq:Y_blocking}, which can be re-written as $q_n=p_{n-1}-p_n$. If the servers are homogeneous then the well-known Erlang Loss Formula can be applied for the blocking probabilities, but this is not possible for an infinite server system with finite capacity. Otherwise, the blocking probabilities are given in \cite{Y1987}:
\begin{equation}\label{eq:blocking}
p_n=\prod_{i=1}^n L_{i-1}(\mu_i)=L_{n-1}(\mu_n)p_{n-1},
\end{equation}
where $p_0:=1$, $L_0(s)$ is the LST of the exogenous inter-arrival distribution, and
\begin{equation}\label{eq:LST}
L_n(s)=\frac{L_{n-1}(\mu_n+s)}{1-L_{n-1}(s)+L_{n-1}(\mu_n+s)}, \quad n\geq 1,
\end{equation}
is the Laplace-Stieljes Transform (LST) of the stationary time between overflows at server $n$, $T_n$. Specifically, $T_n$ is the time between two consecutive arrivals that find the first $n$ servers busy, and recall that $T_0$ is the external inter-arrival time. The recursive formula \eqref{eq:LST} generally relies on a Palm-type theorem for the renewal process of overflows at station $n$, that is, the probability at overflow times as opposed to the time-average distribution which is different as the counting process of overflows is not Poisson. The original result for homogeneous servers appeared in \cite{T1959} (see also p.37 of \cite{book_R1962}). A generalization to heterogeneous service rates appeared in \cite{TN1975} and a similar model with an additional waiting buffer for customers that find all servers busy upon arrival \cite{C1976}.

The derivative of \eqref{eq:LST} is
\begin{equation}\label{eq:LST_deriv}
L_n'(s)= \frac{L_{n-1}'(\mu_{n}+s)(1-L_{n-1}(s))+L_{n-1}'(s)L_{n-1}(\mu_{n}+s)}{\left(1-L_{n-1}(s)+L_{n-1}(\mu_n+s)\right)^2},
\end{equation}
and thus the mean time between overflows from server $n$ is
\[
\E T_n=-L_n'(0)=\frac{\E T_{n-1}}{L_{n-1}(\mu_n)}=\frac{\E T_{n-2}}{L_{n-2}(\mu_{n-1})L_{n-1}(\mu_n)}=\cdots=\frac{\E T_0}{p_n},
\]
yielding
\begin{equation}\label{eq:ET_n}
\E T_n=\frac{1}{\lambda p_{n}}, \quad n\geq 1.
\end{equation}

We next state a technical lemma that will be useful in the following analysis. We do not prove the lemma as these properties are straightforward extensions or rephrasing of known results.

\begin{lemma}\label{lemma:LST_properties}
The functions $L_i(s)$ satisfy the following properties:
\begin{enumerate}[label=\alph*.]
\item $L_n(s)$ is strictly decreasing with $s$, $L_n(0)=1$ and $\lim_{s\to\infty}L_n(s)=0$.
\item $L_{n-1}(s) > L_{n}(s)$ for any $s>0$ and $n\geq 1$ (Proof in \cite{Y1987}).
\item If $\mu_1\geq\mu_2\geq\cdots$ then the decreasing sequence $p_n$ is convex in the discrete sense: $p_{n-1}+p_{n+1}> 2p_n,\quad \forall n\geq 2$ (Proof in \cite{Y1986}).
\end{enumerate}
\end{lemma}

\begin{remark*}
In the following sections we use the notation $a\wedge b:=\min\{a,b\}$. We will also make frequent use of the notation $a_n\approx b_n$ to indicate that the two sequence, $\{a_n\}$ and $\{b_n\}$, have the same tail behaviour. Formally, this means that there exists a constant $0<C<\infty$ such that
\[
\lim_{n\to\infty}\frac{a_n}{b_n}=C,
\]
and when used for the limits themselves it means that they are proportional:
\[
\frac{\lim_{n\to\infty}a_n}{\lim_{n\to\infty}b_n}=C.
\]
We use $a_n\ll b_n$ to indicate that $C=0$. In the numerical analysis we will use $\eqsim$ when numerical results are close (in a non-accurate sense) to some value.
\end{remark*}

\section{System stability}\label{sec:stability}
In an infinite-server system with infinite capacity, as the homogeneous server GI/M/$\infty$ model, the system is always stable. In many queueing systems with finite capacity the queue-length process is stable if $\rho<1$, in the sense that the number of customers in the system does not explode (and the underlying embedded Markov process is positive recurrent). However this is not sufficient for our model because the service allocation may cause the effective arrival rate to a subset of the system to exceed its service capacity. For example if the arrival process is Poisson then $p_1=L_0(\mu_1)=\frac{\lambda}{\lambda+\mu_1}$, further if $\mu_1=\mu-\epsilon$ then the effective arrival rate to the system excluding server $1$ is $\lambda p_1=\lambda\frac{\lambda}{\lambda+\mu-\epsilon}$, which is larger than $\epsilon$ when $\epsilon$ is chosen small enough.

This paper does not directly address the issue of positive recurrence of the process $\mathbf{X}$ at arrival times, which seems to require a different approach than the blocking probability and delay computations employed here. Rather, we define two different levels of stability: the first simply states that all subsystems have a greater capacity than their effective arrival rate, and the second is finite expected delay - a condition that may not hold even if the underlying process is positive recurrent. 

As a first reasonable condition, and as we will show in Proposition \ref{prop:equivalence}, one that is also necessary for finite expected delay, we would like a service-rate sequence $\{\mu_n\}_{n=1}^\infty$ to satisfy
\begin{equation}\label{eq:stability}
\lambda p_n<\mu-\sum_{i=1}^n\mu_i=\sum_{i=n+1}^\infty\mu_i, \quad \forall n\geq 1.
\end{equation}

That is, the effective arrival rate into the system excluding the first $n$ servers is smaller than the remaining service capacity. In the memoryless arrival example, the first condition for $n=1$ is $\lambda \frac{\lambda}{\lambda+\mu_1}<\mu-\mu_1$, or equivalently,
\begin{equation*}\label{eq:mu1_stability}
\mu_1\in\left(0,\frac{1}{2}\left(\mu-\lambda+\sqrt{(\mu-\lambda)(\mu+3\lambda)}\right)\right).
\end{equation*}

\begin{definition}
A service-rate sequence $\{\mu_n\}_{n=1}^\infty$ is feasible if it satisfies condition \eqref{eq:stability}.
\end{definition}

Denote the set of feasible service-rate sequence by 
\[
\mathcal{M}:=\left\lbrace \{\mu_n\}_{n=1}^\infty: \quad  \lambda p_n<\sum_{i=n+1}^\infty \mu_i, \ \forall n\geq 1 \right\rbrace.
\]

\begin{proposition}\label{lemma:feasible}
For any $\lambda\in(0,\mu)$ there exists a feasible service-rate sequence ($\mathcal{M}\neq\emptyset$) that is decreasing and satisfies $\sum_{n=1}^\infty \mu_n=\mu$ .
\end{proposition}
\begin{proof}
If $\lambda<\mu$ then the range for $\mu_1$ given by
\[
\mu_1<\mu-\lambda p_1=\mu-\lambda L_0(\mu_1).
\] 
Let $m_1$ be the solution to $\mu_1=\mu-\lambda L_0(\mu_1)$. Recall that $L_0(0)=1$ and that $\lambda<\mu$, therefore as $L_0$ is an LST, hence convex, there exists a unique solution $m_1>0$. This argument is illustrated in Figure \ref{fig:stability_solution}. Any point $\mu_1$ in the interior of the interval $(0,m_1)$ satisfies the stability condition for $n=1$, in particular $\tilde{\mu}_1=\alpha m_1$, for any $\alpha\in(0,1)$.

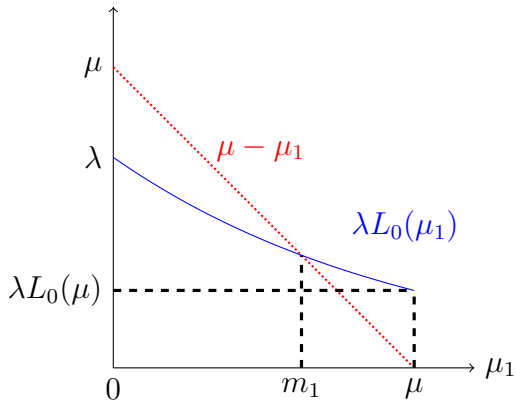
\begin{figure}[H]
\centering
\begin{tikzpicture}[xscale=4,yscale=4]
  \def\xmin{0}
  \def\xmax{1.2}
  \def\ymin{0}
  \def\ymax{1.2}
    \draw[->] (\xmin,\ymin) -- (\xmax,\ymin) node[right] {$\mu_1$} ;
    \draw[->] (\xmin,\ymin) -- (\xmin,\ymax);

	\node at (\xmin,\ymin) [below] {$0$};
	\node at (1,\ymin) [below] {$\mu$};
	\node at (0.6255,\ymin) [below] {$m_1$};
	\node at (\xmin,1) [left] {$\mu$};
	\node at (\xmin,0.7) [left] {$\lambda$};
	\node at (\xmin,0.2575156) [left] {$\lambda L_0(\mu)$};
	
    \draw[red,thick,densely dotted, domain=0:1] plot (\x, {1-\x});
    \draw[blue, domain=0:1] plot (\x, {0.7*exp(-\x)});
    
    \draw[black, very thick, dashed] (0.6255,\ymin)--(0.6255,0.3745);
    \draw[black, very thick, dashed] (\xmin,0.2575156)--(1,0.2575156);
    \draw[black, very thick, dashed] (1,\ymin)--(1,0.2575156);
    
    \node[draw=none,,color=red] at (0.49,0.73) {$\mu-\mu_1$};
    \node[draw=none,,color=blue] at (0.97,0.47) {$\lambda L_0(\mu_1)$};
\end{tikzpicture}
 \caption{Illustration of the feasibility interval $(0,m_1)$ for $\mu_1$.}\label{fig:stability_solution}
\end{figure}

Suppose that $\tilde{\mu}_1,\ldots,\tilde{\mu}_n$ is a decreasing sequence that satisfies \eqref{eq:stability}, then by applying the product form of \eqref{eq:blocking} we have that condition \eqref{eq:stability} is satisfied for $n+1$ if
\[
\lambda L_{n}(\mu_{n+1})p_{n}<\mu-\sum_{i=1}^n \mu_i-\mu_{n+1},
\]
or equivalently
\[
0< \mu_{n+1} < \mu-\sum_{i=1}^n\mu_i-\lambda p_{n}L_{n}(\mu_{n+1}).
\]

Let $m_{n+1}>0$ be the unique solution to $\mu_{n+1} = \mu-\sum_{i=1}^n\mu_i-\lambda p_{n}L_{n}(\mu_{n+1})$. Repeating the argument illustrated in Figure \ref{fig:stability_solution}, the solution is unique and positive because $L_{n}(\mu_{n+1})\in[0,1]$ is convex and condition \eqref{eq:stability} is satisfied for $n$. We thus conclude that any $\mu_{n+1}\in(0,m_{n+1})$ satisfies \eqref{eq:stability}. In particular, for any $\alpha\in(0,1)$ $\tilde{\mu}_{n+1}:=\alpha(m_{n+1}\wedge\mu_{n})$ is non-increasing, feasible and $\sum_{n=1}^\infty \tilde{\mu}_{n}\leq \mu$.
\end{proof}

Recall that regardless of whether the service-rate sequence is feasible, the sub-system of the first $n$ servers is ergodic for every finite $n$, hence the limit probabilities $p_n$ and $q_n$ exist for any service-rate sequence. Further observe that if condition \eqref{eq:stability} is satisfied for some $N$ then it is satisfied for all $n<N$ as well, as if this was not the case, i.e., $\lambda L_{n-1}(0)p_{n-1}>\mu-\sum_{i=1}^{n-1}\mu_i$, then there is no $\mu_n>0$ such that $\lambda L_{n-1}(\mu_n)p_{n-1}=\mu-\sum_{i=1}^{n-1}\mu_i-\mu_n$ (consider Figure \ref{fig:stability_solution} for the case that the solid convex overflow rate line starts above the dotted linear capacity allocation line).

\begin{lemma}\label{lemma:L_limit}
For any service-rate sequence $\{\mu_n\}_{n=1}^\infty$ such that $\mu_n>0$ for all $n\geq 1$ and $\sum_{n=1}^\infty \mu_n=\mu$, there exists a limit
\[
\ell:=\lim_{n\to\infty}L_{n-1}(\mu_n)\in\left(L_0(\mu),1\right].
\]
\end{lemma}
\begin{proof} 
Iterating the recursion of \eqref{eq:LST} yields
\begin{equation}\label{eq:Ln_product}
\begin{split}
L_{n-1}(\mu_n) &= \frac{L_{n-2}(\mu_n+\mu_{n-1})}{1-L_{n-2}(\mu_n)+L_{n-2}(\mu_n+\mu_{n-1})} \\
&= \frac{L_{n-3}(\mu_n+\mu_{n-1}+\mu_{n-2})}{\left[1-L_{n-3}(\mu_n+\mu_{n-1})+L_{n-3}(\mu_n+\mu_{n-1}+\mu_{n-2})\right]\left[1-L_{n-2}(\mu_n)+L_{n-2}(\mu_n+\mu_{n-1})\right]} \\
&= \frac{L_{0}\left(\sum_{i=1}^{n}\mu_{i}\right)}{\prod_{i=1}^{n-1}\left[1-L_{i-1}\left(\sum_{k=i+1}^{n}\mu_k\right)+L_{i-1}\left(\sum_{k=i}^{n}\mu_k\right)\right]}.
\end{split}
\end{equation}
By Lemma \ref{lemma:LST_properties}a each term in the product in the denominator is smaller than one, therefore for any $n\geq 1$, $L_{n-1}(\mu_n)\geq L_{0}\left(\sum_{i=1}^{n}\mu_{i}\right)$. The lower and upper bounds are obtained using Lemma \ref{lemma:LST_properties}a: first the fact that the LST is a decreasing function yields $L_{n-1}(\mu_n)\geq L_0\left(\sum_{i=1}^n\mu_i\right)\geq L_0(\mu)$, and furthermore every every term in the sequence $L_{n-1}(\mu_n)$ is bounded from above by $1$, hence,
\[
L_0(\mu)\leq L_{n-1}(\mu_n)\leq 1.
\]
By the continuity of $L_0$, $\lim_{n\to\infty}L_{0}\left(\sum_{i=1}^{n}\mu_{i}\right)=L_0(\mu)$, since $\sum_{i=1}^\infty\mu_i=\mu$. Let 
\[
a_{in}:=1-L_{i-1}\left(\sum_{k=i+1}^{n}\mu_k\right)+L_{i-1}\left(\sum_{k=i}^{n}\mu_k\right),
\]
then
\[
L_{n-1}(\mu_n) = \frac{L_{0}\left(\sum_{i=1}^{n}\mu_{i}\right)}{\prod_{i=1}^{n-1}a_{in}}\leq 1,
\]
which implies that the product in the denominator does not converge to zero. If the limit $a:=\lim_{n\to\infty} \prod_{i=1}^{n-1}a_{in}>0$ exists then
\[
\ell=\lim_{n\to\infty}L_{n-1}(\mu_n)=\frac{L_0(\mu)}{a}.
\]
We will verify the existence of the limit $a$ in three steps as outlined below:
\begin{enumerate}
\item We show that $a_{in}\in(0,1]$ is increasing with $n$ and therefore has a limit $a_{i}:=\lim_{n\to\infty}a_{in}$.
\item Let $b_{in}:=|\log(a_{in})|$, then $b_{in}\in[0,\infty)$ is decreasing with $n$ and has a limit $b_{i}:=\lim_{n\to\infty}b_{in}$.
\item The sum $\sum_{i=1}^{n-1}b_{in}=-\log\left(\prod_{i=1}^{n-1}a_{in}\right)$ converges to a limit $b<\infty$, and therefore the product $\prod_{i=1}^{n-1}a_{in}$ converges to a limit $a=e^{-b}$.
\end{enumerate}
\textbf{Step 1:} Since $\mu_n>0$ for all $n\geq 1$ then the convexity of $L_{i-1}$ implies that
\[
L_{i-1}\left(\sum_{k=i+1}^{n}\mu_k\right)-L_{i-1}\left(\sum_{k=i}^{n}\mu_k\right)>L_{i-1}\left(\sum_{k=i+1}^{n+1}\mu_k\right)-L_{i-1}\left(\sum_{k=i}^{n+1}\mu_k\right)>0, \quad \forall 1\leq i\leq n<\infty,
\]
hence, $a_{in}$ is increasing with $n$ and bounded by one, and thus there exists a limit,
\[
a_{i}:=\lim_{n\to\infty}a_{in}=1-L_{i-1}\left(\sum_{k=i+1}^{\infty}\mu_k\right)+L_{i-1}\left(\sum_{k=i}^{\infty}\mu_k\right).
\]

\textbf{Step 2:} Let $b_{in}:=|\log(a_{in})|$ and $b_{i}:=\lim_{n\to\infty}b_{in}$, and observe that $\log(a_{in})\leq 0$ as $a_{in}\in[L_0(\mu),1]$. The sum $\sum_{i=1}^{n-1}b_{in}=-\log\left(\prod_{i=1}^{n-1}a_{in}\right)$ is bounded as $n\to\infty$ because the product does not converge to zero. Moreover, as $a_{in}$ is increasing with $n$, $b_{in}=|\log(a_{in})|$ is decreasing and $b_{in}\geq b_i$ for all $1\leq i\leq n$. 

\textbf{Step 3:} The monotonicity of $b_{in}$ further implies that $\sum_{i=1}^{n-1}b_{in}\geq \sum_{i=1}^{n-1}b_i $. As argued in the previous step, the series $\sum_{i=1}^{n-1}b_{in}$ is bounded when taking $n\to\infty$, hence $\sum_{i=1}^{n-1}b_{i}$ is bounded and increasing and thus converges to a limit $b<\infty$. This further implies that for every $\epsilon>0$ there exists an $N$ such that $\sum_{i=N}^\infty b_i<\frac{\epsilon}{2}$. As $\sum_{i=N}^{n-1}b_{in}\leq \sum_{i=N}^{\infty}b_{in}$, the monotone convergence theorem yields
\[
\lim_{n\to\infty}\sum_{i=N}^{n-1}b_{in}\leq \lim_{n\to\infty}\sum_{i=N}^{\infty}b_{in}=\sum_{i=N}^\infty b_i<\frac{\epsilon}{2},
\] 
and we conclude that there exists some $\hat{N}\geq N$ such that $\sum_{i=N}^{n-1}b_{in}\leq \epsilon$, for all $n>\hat{N}$. Therefore, for $n>\hat{N}$,
\[
\sum_{i=1}^{n-1}b_{in}=\sum_{i=1}^{N-1}b_{in}+\sum_{i=N}^{n-1}b_{in}\leq \sum_{i=1}^{N-1}b_{in}+\epsilon.
\]
Furthermore, for any $N$ there exists some $\tilde{N}\geq N$ such that $\sum_{i=1}^{N-1}b_{in}-\sum_{i=1}^{N-1}b_{i}\leq\epsilon$, for all $n\geq\tilde{N}$, and then for all $n\geq\max\{\tilde{N},\hat{N}\}$,
\[
\sum_{i=1}^{n-1}b_{in}-\sum_{i=1}^{n-1}b_{i}\leq \sum_{i=1}^{N-1}b_{in}+\epsilon-\sum_{i=1}^{n-1}b_{i}\leq \sum_{i=1}^{N-1}b_{in}-\sum_{i=1}^{N-1}b_{i} +\epsilon \leq 2\epsilon,
\]
which yields
\[
0\leq \sum_{i=1}^{n-1}b_{in}-\sum_{i=1}^{n-1}b_{i}\leq 2\epsilon, \quad \forall n\geq \max\{\tilde{N},\hat{N}\}.
\]
The above holds for an arbitrary $\epsilon>0$ and thus we conclude that $\lim_{n\to\infty}\sum_{i=1}^{n-1}b_{in}=\sum_{i=1}^{\infty}b_{i}$, and
\[
\prod_{i=1}^{n-1}a_{in}=e^{-\sum_{i=1}^{n-1}b_{in}}\xrightarrow{n\to\infty}e^{-\sum_{i=1}^{\infty}b_{i}}=:a.
\]
We conclude that there exists a limit $\ell=\lim_{n\to\infty}L_{n-1}(\mu_n)=\frac{L_0(\mu)}{a}$. 
\end{proof}

We now turn our attention to the expected delay,
\[
\E S=\sum_{n=1}^\infty \frac{q_n}{\mu_n}.
\]

\begin{definition}
A service-rate sequence $\{\mu_n\}_{n=1}^\infty$ satisfies finite delay (FD) if it belongs to
\[
\mathcal{FD}:=\left\lbrace\{\mu_n\}_{n=1}^\infty: \ \E S<\infty\right\rbrace.
\]
\end{definition}

From \eqref{eq:Y_blocking} and \eqref{eq:blocking} we have
\begin{equation}\label{eq:qn_recursion}
q_n=p_{n-1}-p_n=(1-L_{n-1}(\mu_n))\prod_{i=1}^{n-1} L_{i-1}(\mu_i),
\end{equation}
that is, a non-homogeneous geometric distribution. For $n\geq 1$, recall that $\ell_n:=L_{n-1}(\mu_n)$, $\ell:=\lim_{n\to\infty}\ell_n$ and denote $p:=\lim_{n\to\infty}p_n$. If $\ell<1$ then the geometric term tends to the constant $\ell$, i.e.\ the tail is as of a memoryless distribution.

\begin{lemma}\label{lemma:L_p_limit}
For any external arrival distribution $T_0$ and service-rate sequence $\{\mu_n\}_{n=1}^\infty$,
\begin{enumerate}[label=\alph*.]
\item $p=0 \ \Leftrightarrow \ \sum_{n=1}^\infty (1-\ell_n)=\infty$,
\item $\ell<1 \ \Rightarrow \ p=0$,
\item $\ell=1 \ \Leftrightarrow$ $Y$ is heavy tailed: $\sum_{n=1}^\infty q_n e^{\eta n}=\infty, \ \forall \eta>0$,
\item $\frac{\mu_n}{p_{n}}\xrightarrow{n\to\infty} 0 \ \Rightarrow \ \ell=1$.
\end{enumerate}
\end{lemma}
\begin{proof} ${ }$
\begin{enumerate}[label=\alph*.]
\item By \eqref{eq:blocking}, $p_n=\prod_{i=1}^n \ell_i$. For any positive sequence $\{a_n\}_{n=1}^\infty$ the convergence of the product $\prod_{n=1}^\infty (1-a_n)$ to a non-zero and finite limit is equivalent to the convergence of the sum $\sum_{n=1}^\infty a_n$ (see \cite{book_A1974}, p. 209), hence $p_n\xrightarrow{n\to\infty}0$ if and only if $\sum_{n=1}^\infty (1-\ell_n)=\infty$.
\item If $\ell<1$ then clearly $\sum_{n=1}^\infty (1-\ell_n)=\infty$, hence by the previous property we have that $p=0$.
\item An equivalent condition for $Y$ being heavy-tailed is given by Theorem 2.6 of \cite{book_SKZ2011}:
\[
-\frac{1}{n}\log\P(Y>n)\xrightarrow{n\to\infty} 0.
\]
As $\P(Y>n)=p_n$, this is equivalent by the Stolz-Ces\'{a}ro Theorem (discrete version of L'Hopital's Rule) to
\[
\log p_n-\log p_{n+1}\xrightarrow{n\to\infty} 0,
\]
and as $p_n=\prod_{n=1}^\infty \ell_n$, we conclude that
\[
\log p_n-\log p_{n+1}=-\log\frac{p_{n+1}}{p_{n}}=-\log\ell_{n+1}
\]
converges to zero if and only if $\ell=1$.
\item Recall the definition of the LST, $\ell_n=L_{n-1}(\mu_n)=\E e^{-\mu_n T_{n-1}}$, then by applying Jensen's inequality and \eqref{eq:ET_n} we conclude that
\[
\ell_n=\E e^{-\mu_n T_{n-1}} \geq e^{-\mu_n \E T_{n-1}}=e^{-\frac{\mu_n}{\lambda p_{n-1}}}\geq e^{-\frac{\mu_n}{\lambda p_{n}}}.
\]
\end{enumerate}
\end{proof}

Lemma \ref{lemma:L_p_limit} suggests that the tail behaviour of the LST sequence $L_{n-1}(\mu_n)$, and its limit in particular, is a key component in analysing the stability and expected delay in the system. 
The following proposition summarizes the relationship between feasibility, finite expected delay and the tail behaviour of the LST sequence. In particular we obtain a necessary and sufficient condition for finite expected delay: any feasible service-rate sequence that satisfies $\ell<1$ with slower decay rate than $\ell$. This will be useful for the optimization problem in the following sections.

\begin{proposition}\label{prop:equivalence}
Let $\lambda$ and $\{\mu_n\}_{n=1}^\infty$ be the arrival rate and service-rate sequence, such that $\sum_{n=1}^\infty \mu_n=\mu>\lambda$. Then the following properties are satisfied:
\begin{enumerate}[label=\alph*.]
\item $ \{\mu_n\}_{n=1}^\infty\in\mathcal{M} \ \Rightarrow \ p=0$,
\item $\{\mu_n\}_{n=1}^\infty\in\mathcal{FD} \ \Rightarrow \ \ell<1$,
\item If $\ell<1$, such that $\ell^n\ll \mu_n$ then $\{\mu_n\}_{n=1}^\infty\in\mathcal{FD}$,
\item $\{\mu_n\}_{n=1}^\infty\in\mathcal{FD} \ \Rightarrow \ \{\mu_n\}_{n=1}^\infty\in\mathcal{M}$ (i.e.\ $\mathcal{FD}\subseteq\mathcal{M}$).
\end{enumerate}
\end{proposition}
\begin{proof} ${}$
\begin{enumerate}[label=\alph*.]
\item This can be seen directly from \eqref{eq:stability} as the right-hand side tends to zero due to the capacity constraint.

\item The series $\sum_{n=1}^\infty\mu_n$ converges, therefore its tail decays faster than that of the harmonic series. Without loss of generality, as we are only interested in the tail behaviour, we assume this is the case for all $n\geq 1$:
\[
\mu_n<\frac{1}{n} \ \Leftrightarrow \ \frac{1}{\mu_n}>n.
\]
The expected delay then satisfies
\[
\E S=\E\frac{1}{\mu_Y}>\E Y.
\]
If $\ell<1$ then $q_n\approx (1-\ell)\ell^{n-1}$ by \eqref{eq:qn_recursion}, and
\[
\E Y=\sum_{n=1}^\infty nq_n \approx \frac{1}{1-\ell}.
\]
Hence, $\ell=1$ implies that $\E S=\infty$. In other words, $\ell<1$ is a necessary condition for finite delay.
\item If $\ell<1$ then by Lemma \ref{lemma:L_p_limit}c $Y$ is not heavy tailed: there exists $\eta>0$ such that
\[
\sum_{n=1}^\infty q_n e^{\eta n}<\infty.
\]
If we further assume that 
\[
\E S=\sum_{n=1}^\infty q_n\frac{1}{\mu_n}=\infty,
\]
then the tail of the service-rate sequence decays even faster than the exponential term, i.e.,
\[
\frac{1}{\mu_n}>e^{\eta n} \ \Leftrightarrow \ \mu_n<e^{-\eta n}.
\]
Equivalently we can say that $\mu_n<\alpha^n$ for $\alpha=e^{-\eta}$. If $\mu_n=\beta^n$ such that $\ell<\beta<\alpha$ then
\[
\E S=\sum_{n=1}^\infty q_n\frac{1}{\mu_n}\approx\sum_{n=1}^\infty \frac{\ell^n}{\beta^n}<\infty,
\]
contradicting the assumption that $\E S=\infty$. Hence, if $\ell<1$ and the service-rate sequence decays slower than $\ell^n$ then the expected delay is finite.
\item Any sequence $\{\mu_n\}_{n=1}^\infty$ that decays at least as fast as $\ell^n$ induces infinite expected delay because 
\[
\sum_{n=1}^\infty\frac{\ell^n}{\mu_n}=\infty.
\]
If $\{\mu_n\}_{n=1}^\infty\in\mathcal{M}$ such that $\mu_n\approx\ell^n$ then 
\[
\sum_{i=n+1}^\infty \mu_i\approx \sum_{i=n+1}^\infty \ell^i=\frac{\ell^{n+1}}{1-\ell}\approx \ell^n \approx \lambda p_n,
\]
hence the tail of the service-rate sequence is on the boundary of the feasible range given by \eqref{eq:stability}. This means that the inequality condition of \eqref{eq:stability} is satisfied for every $n$ although the difference converges to zero, and moreover that any sequence decreasing at a faster rate is not feasible.
We conclude that feasibility of a sequence, $\{\mu_n\}_{n=1}^\infty\in\mathcal{M}$, is a necessary condition for finite delay, $\{\mu_n\}_{n=1}^\infty\in\mathcal{FD}$.
\end{enumerate}
\end{proof}

Proposition \ref{prop:equivalence} yields a convenient necessary and sufficient condition for a feasible service-rate sequence to satisfy $\E S<\infty$, by combining parts b and c of the proposition:
\begin{equation}\label{eq:finite_condition}
\sum_{n=1}^\infty\frac{\ell^n}{\mu_n}<\infty .
\end{equation}

We conclude this section by pointing out open questions and additional refinements of the stability analysis that can be considered in future work on this model.

\begin{remark}\label{remark:conj_l1_EW}
We conjecture that a stronger result than Proposition \ref{prop:equivalence} holds, namely that 
\[
\{\mu_n\}_{n=1}^\infty\in\mathcal{FD} \ \Leftrightarrow \ \ell<1.
\]
Proposition \ref{prop:equivalence}c establishes that if $\mu_n$ decays slowly enough then $\ell<1$ is sufficient for FD. Furthermore, if $\ell<1$ and $\E S=\infty$ such that $\mu_n\leq \beta^n$, where $\beta<\ell$, then $\frac{\mu_n}{p_n}\leq\frac{\beta^n}{\ell^n}\to 0$, and by Lemma \ref{lemma:L_p_limit}d we conclude that $\ell_n\to 1$, a contradiction. That is, if the service-rate sequence decays faster than the blocking probability then $\ell=1$ and the expected delay is infinite. We are left with checking the case of $\mu_n\approx p_n\approx\ell^n$, where $\ell<\beta$. We believe that in this case $\ell_n\to 1$ as well, and this belief is supported by numerical tests, but we were unable to prove this claim. In such a case $\ell_n\to 1$ at a slow rate (in the sense of Lemma \ref{lemma:L_p_limit}a). If this is true then indeed $\E S<\infty \Leftrightarrow \ell<1$, but we leave this issue as an open question. A more speculative conjecture is that the extreme case on the boundary of the feasibility region, $\mu_n\approx\ell^n$, occurs when the underlying process is null-recurrent.
\end{remark}

\begin{remark}\label{remark:FD_exist}
An additional open question is whether for any $\lambda<\mu$ there exists a feasible service-rate sequence $\{\mu_n\}_{n=1}^\infty$ such that $\E S<\infty$, i.e.\ $\mathcal{FD}\neq\emptyset$. We conjecture that this is the case, but have no proof. Note that for any finite $N$ it is possible to construct a sequence $\{\mu\}_{n=1}^N$ such that $\mu_n$ decays at a slower rate than $p_n$ (by some positive factor), but the difficulty lies in showing that the rates don't coincide when taking $N\to\infty$.  
\end{remark}

\begin{remark}\label{remark:fixed_blocking}
Little's Law implies that a finite expected delay, $\E S<\infty$, is equivalent to a finite expected number of customers in the system. This means that a feasible service-rate sequence and $\ell<1$ are both necessary, but not sufficient, conditions for the expected number of customers in the system to be finite, which in itself is sufficient but not necessary for general system stability (in terms of positive recurrence of the underlying process). Nevertheless, the probability that a customer that arrives at server $n$, after being blocked by the previous servers, finds it busy is $\P(X_n=1|Y\geq n)=L_{n-1}(\mu_n)$. This can be seen by considering the blocking probability of the first server in a system with external arrival distribution $L_{n-1}$. From Lemma \ref{lemma:L_p_limit}b we have that for any service-rate sequence,
\[
\P(X_n=1|Y\geq n)\xrightarrow{n\to\infty}\ell>0.
\]
This gives us an interesting result: ``bad'' servers, i.e.\ large $n$ and slow service-rate $\mu_n$, block a fixed proportion of arrivals to them.
\end{remark}

\section{Geometric service-rate series}\label{sec:geometric}

A very natural capacity allocation to consider is using a simple geometric sequence determined by a single parameter. This is especially called for in light of Proposition \ref{prop:equivalence} that established that if $\ell<1$ and the service-rate sequence decays slower than a geometric sequence with rate $\ell$ then the expected delay is finite. Moreover, such service-rate sequences satisfy properties that will be useful for dealing with the capacity allocation problem. Namely, the stability and finite delay conditions have a simple form and the tail of the optimal solution is indeed approximately geometric under some invariance conditions which will be elaborated.

Suppose that the service-rate sequence is determined by a single parameter representing the service capacity allocated to the first server. If we assume without loss of generality that $\mu=1$ (and then $\rho=\lambda$), then the class of such service-rate sequence is
\[
\mathcal{M}_{g}:=\left\lbrace\{\mu_n\}_{n=1}^\infty:\ \mu_n=\alpha(1-\alpha)^{n-1}, \alpha\in(0,1)\right\rbrace.
\]
In this formulation, the single parameter is the service allocation of the first server, $\mu_1=\alpha$. 

For any $\{\mu_n\}\in\mathcal{M}_{g}$ we have that $\sum_{i=n+1}^\infty\mu_i=(1-\alpha)^n$, and therefore the feasibility condition \eqref{eq:stability} is simply
\begin{equation*}\label{eq:stability_alpha}
\lambda p_n< (1-\alpha)^{n},\quad \forall n\geq 1,
\end{equation*}
and the finite delay condition \eqref{eq:finite_condition} is
\begin{equation}\label{eq:stability_alpha}
\lambda p_n\ll (1-\alpha)^{n} \ \Leftrightarrow \ \ell<1-\alpha .
\end{equation}
It is possible that the feasibility condition is met but $p_n\to (1-\alpha)$ (from below) and then condition \eqref{eq:stability_alpha} is not met.

Let $\ell_n(\alpha):=L_{n-1}\left(\alpha(1-\alpha)^{n-1}\right)$ and $\ell(\alpha)=\lim_{n\to\infty}\ell_n(\alpha)$. In Figure \ref{fig:l_alpha} the sequence of functions $\ell_n(\alpha)$ are illustrated for the case of Poisson arrivals and $\lambda=\rho=0.2$. There are several interesting observations to be made from this figure, all of which are robust for different values of $\rho$ and other external inter-arrival distributions. For every $n\geq 2$, the function $\ell_n(\alpha)$ is unimodal (attaining a minimum) and $\ell_n(0)=\ell_n(1)=1$. Furthermore, the slope of the functions at $\ell_n(0)$ is decreasing with $n$, which can be verified by recalling that the derivative of the LST at zero equals the negative of the overflow expectation given by \eqref{eq:ET_n}: $\E T_n=\frac{1}{\lambda p_n}$ (which goes to $-\infty$ as $n\to\infty$ and explains why there seems to be a downwards discontinuity at zero as $n$ gets large). This implies that for every $n\geq 1$ there exists an $\alpha\in(0,1)$ such that $\ell_n(\alpha)<1-\alpha$. It appears that this is the case also for the limit $\ell(\alpha)$, which implies the finite delay condition \eqref{eq:stability_alpha}, but we currently have no proof to this effect. A proof of this would resolve the open question described in Remark \ref{remark:FD_exist} in the previous section. The limiting function $\ell(\alpha)$ appears to have an invariance region, in which the value of the function is almost constant with respect to $\alpha$, specifically: the function starts at $\ell(0)=1$, sharply decreases after zero, has an interval $\alpha\in(0,\overline{\alpha})$ which it is almost constant $\ell(\alpha)\eqsim\overline{\ell}$, and then sharply increases back to $\ell(\alpha)\eqsim 1$ for $\alpha\in[\overline{\alpha},1]$. In the case of a Poisson arrival process we observe that $\overline{\ell}=\rho$ and $\overline{\alpha}=1-\sqrt{\rho}$. The latter value is the explicit solution of $\ell_1(\alpha)=\ell_2(\alpha)$, i.e.\ the $\alpha$ value where the first and second functions intersect. Interestingly, it appears that all of the functions intersect at around the same point. It is hard to tell whether the limiting function $\ell(\alpha)$ would have an upward discontinuity to $1$ at $\overline{\alpha}$ or just a sharp and continuous increase as we see for $n=25$. We were unable to computationally explore the function for higher values.

\begin{figure}[H]
\centering
\begin{tikzpicture}[xscale=9,yscale=9]
  \def\xmin{0}
  \def\xmax{1.05}
  \def\ymin{0}
  \def\ymax{1.01}
    \draw[->] (\xmin,\ymin) -- (\xmax,\ymin) node[right] {$\alpha$} ;
    \draw[->] (\xmin,\ymin) -- (\xmin,\ymax) node[above] {$\ell_n(\alpha)$} ;
    \foreach \x in {0,0.1,0.2,0.3,0.4,0.5,0.6,0.7,0.8,0.9,1}
    \node at (\x,\ymin) [below] {\x};
    \foreach \y in {0,0.5,1}
    \node at (\xmin,\y) [left] {\y};
    \node at (\xmin,0.14) [left] {\scriptsize{$L_0(\mu)$}};
    \node at (\xmin,0.22) [left] {\scriptsize{$\rho$}};
    
     \draw[smooth,loosely dashed, thick ] (\xmin,0.166667)--(1,0.166667);
     \draw[smooth,loosely dashed, thick ] (\xmin,0.2)--(1,0.2);
    
    \draw[densely dashed,green] (0,1)--	(0.01,0.952381)--	(0.02,0.9090909)--	(0.03,0.8695652)--	(0.04,0.8333333)--	(0.05,0.8)--	(0.06,0.7692308)--	(0.07,0.7407407)--	(0.08,0.7142857)--	(0.09,0.6896552)--	(0.1,0.6666667)--	(0.11,0.6451613)--	(0.12,0.625)--	(0.13,0.6060606)--	(0.14,0.5882353)--	(0.15,0.5714286)--	(0.16,0.5555556)--	(0.17,0.5405405)--	(0.18,0.5263158)--	(0.19,0.5128205)--	(0.2,0.5)--	(0.21,0.4878049)--	(0.22,0.4761905)--	(0.23,0.4651163)--	(0.24,0.4545455)--	(0.25,0.4444444)--	(0.26,0.4347826)--	(0.27,0.4255319)--	(0.28,0.4166667)--	(0.29,0.4081633)--	(0.3,0.4)--	(0.31,0.3921569)--	(0.32,0.3846154)--	(0.33,0.3773585)--	(0.34,0.3703704)--	(0.35,0.3636364)--	(0.36,0.3571429)--	(0.37,0.3508772)--	(0.38,0.3448276)--	(0.39,0.3389831)--	(0.4,0.3333333)--	(0.41,0.3278689)--	(0.42,0.3225806)--	(0.43,0.3174603)--	(0.44,0.3125)--	(0.45,0.3076923)--	(0.46,0.3030303)--	(0.47,0.2985075)--	(0.48,0.2941176)--	(0.49,0.2898551)--	(0.5,0.2857143)--	(0.51,0.2816901)--	(0.52,0.2777778)--	(0.53,0.2739726)--	(0.54,0.2702703)--	(0.55,0.2666667)--	(0.56,0.2631579)--	(0.57,0.2597403)--	(0.58,0.2564103)--	(0.59,0.2531646)--	(0.6,0.25)--	(0.61,0.2469136)--	(0.62,0.2439024)--	(0.63,0.2409639)--	(0.64,0.2380952)--	(0.65,0.2352941)--	(0.66,0.2325581)--	(0.67,0.2298851)--	(0.68,0.2272727)--	(0.69,0.2247191)--	(0.7,0.2222222)--	(0.71,0.2197802)--	(0.72,0.2173913)--	(0.73,0.2150538)--	(0.74,0.212766)--	(0.75,0.2105263)--	(0.76,0.2083333)--	(0.77,0.2061856)--	(0.78,0.2040816)--	(0.79,0.2020202)--	(0.8,0.2)--	(0.81,0.1980198)--	(0.82,0.1960784)--	(0.83,0.1941748)--	(0.84,0.1923077)--	(0.85,0.1904762)--	(0.86,0.1886792)--	(0.87,0.1869159)--	(0.88,0.1851852)--	(0.89,0.1834862)--	(0.9,0.1818182)--	(0.91,0.1801802)--	(0.92,0.1785714)--	(0.93,0.1769912)--	(0.94,0.1754386)--	(0.95,0.173913)--	(0.96,0.1724138)--	(0.97,0.1709402)--	(0.98,0.1694915)--	(0.99,0.1680672)--	(1,0.168);

    \draw[dotted,blue] (0,1)--	(0.01,0.9506984)--	(0.02,0.9034033)--	(0.03,0.8586987)--	(0.04,0.8168501)--	(0.05,0.7779197)--	(0.06,0.7418449)--	(0.07,0.7084919)--	(0.08,0.6776894)--	(0.09,0.6492504)--	(0.1,0.622986)--	(0.11,0.5987133)--	(0.12,0.5762593)--	(0.13,0.5554638)--	(0.14,0.5361797)--	(0.15,0.5182731)--	(0.16,0.5016225)--	(0.17,0.4861183)--	(0.18,0.4716618)--	(0.19,0.4581639)--	(0.2,0.4455446)--	(0.21,0.4337317)--	(0.22,0.4226604)--	(0.23,0.4122722)--	(0.24,0.4025141)--	(0.25,0.3933386)--	(0.26,0.3847026)--	(0.27,0.3765669)--	(0.28,0.3688963)--	(0.29,0.3616586)--	(0.3,0.3548248)--	(0.31,0.3483681)--	(0.32,0.3422646)--	(0.33,0.3364921)--	(0.34,0.3310306)--	(0.35,0.3258617)--	(0.36,0.3209687)--	(0.37,0.3163362)--	(0.38,0.3119502)--	(0.39,0.3077978)--	(0.4,0.3038674)--	(0.41,0.3001481)--	(0.42,0.2966302)--	(0.43,0.2933046)--	(0.44,0.2901632)--	(0.45,0.2871984)--	(0.46,0.2844036)--	(0.47,0.2817725)--	(0.48,0.2792996)--	(0.49,0.27698)--	(0.5,0.2748092)--	(0.51,0.2727832)--	(0.52,0.2708987)--	(0.53,0.2691528)--	(0.54,0.2675429)--	(0.55,0.2660671)--	(0.56,0.264724)--	(0.57,0.2635123)--	(0.58,0.2624317)--	(0.59,0.2614819)--	(0.6,0.2606635)--	(0.61,0.2599774)--	(0.62,0.2594252)--	(0.63,0.2590089)--	(0.64,0.2587315)--	(0.65,0.2585962)--	(0.66,0.2586075)--	(0.67,0.2587704)--	(0.68,0.259091)--	(0.69,0.2595764)--	(0.7,0.2602348)--	(0.71,0.2610761)--	(0.72,0.2621114)--	(0.73,0.2633539)--	(0.74,0.2648186)--	(0.75,0.2665234)--	(0.76,0.2684886)--	(0.77,0.2707385)--	(0.78,0.273301)--	(0.79,0.2762095)--	(0.8,0.2795031)--	(0.81,0.2832285)--	(0.82,0.2874417)--	(0.83,0.29221)--	(0.84,0.2976157)--	(0.85,0.3037593)--	(0.86,0.310766)--	(0.87,0.3187927)--	(0.88,0.3280389)--	(0.89,0.3387625)--	(0.9,0.3513022)--	(0.91,0.366112)--	(0.92,0.3838143)--	(0.93,0.4052866)--	(0.94,0.4318063)--	(0.95,0.4653076)--	(0.96,0.5088659)--	(0.97,0.5676848)--	(0.98,0.6513213)--	(0.99,0.7794425)--	(1,1);

      \draw[densely dashed,red] (0,1)--	(0.01,0.9448419)--	(0.02,0.8817015)--	(0.03,0.8152662)--	(0.04,0.7501043)--	(0.05,0.6893876)--	(0.06,0.6347124)--	(0.07,0.5864956)--	(0.08,0.5444608)--	(0.09,0.5080059)--	(0.1,0.4764279)--	(0.11,0.4490386)--	(0.12,0.4252174)--	(0.13,0.4044264)--	(0.14,0.3862104)--	(0.15,0.370188)--	(0.16,0.3560417)--	(0.17,0.3435074)--	(0.18,0.3323652)--	(0.19,0.3224316)--	(0.2,0.3135528)--	(0.21,0.3055995)--	(0.22,0.2984626)--	(0.23,0.2920497)--	(0.24,0.2862818)--	(0.25,0.2810918)--	(0.26,0.2764219)--	(0.27,0.2722224)--	(0.28,0.2684505)--	(0.29,0.2650691)--	(0.3,0.2620461)--	(0.31,0.2593536)--	(0.32,0.2569674)--	(0.33,0.2548667)--	(0.34,0.2530333)--	(0.35,0.2514516)--	(0.36,0.2501082)--	(0.37,0.2489918)--	(0.38,0.2480929)--	(0.39,0.2474037)--	(0.4,0.246918)--	(0.41,0.2466312)--	(0.42,0.24654)--	(0.43,0.2466426)--	(0.44,0.2469388)--	(0.45,0.2474297)--	(0.46,0.2481178)--	(0.47,0.2490074)--	(0.48,0.2501043)--	(0.49,0.2514162)--	(0.5,0.2529527)--	(0.51,0.2547255)--	(0.52,0.256749)--	(0.53,0.2590399)--	(0.54,0.2616184)--	(0.55,0.264508)--	(0.56,0.2677364)--	(0.57,0.271336)--	(0.58,0.2753446)--	(0.59,0.2798066)--	(0.6,0.2847737)--	(0.61,0.2903067)--	(0.62,0.2964768)--	(0.63,0.3033671)--	(0.64,0.3110756)--	(0.65,0.3197168)--	(0.66,0.3294254)--	(0.67,0.3403584)--	(0.68,0.3526992)--	(0.69,0.36666)--	(0.7,0.3824843)--	(0.71,0.4004472)--	(0.72,0.4208524)--	(0.73,0.444024)--	(0.74,0.4702887)--	(0.75,0.4999447)--	(0.76,0.5332132)--	(0.77,0.5701682)--	(0.78,0.6106479)--	(0.79,0.6541577)--	(0.8,0.6997927)--	(0.81,0.7462202)--	(0.82,0.7917651)--	(0.83,0.8346171)--	(0.84,0.8731272)--	(0.85,0.9061012)--	(0.86,0.9329859)--	(0.87,0.9538818)--	(0.88,0.9694007)--	(0.89,0.9804442)--	(0.9,0.9879912)--	(0.91,0.9929479)--	(0.92,0.996071)--	(0.93,0.9979488)--	(0.94,0.9990156)--	(0.95,0.9995789)--	(0.96,0.9998475)--	(0.97,0.9999577)--	(0.98,0.9999928)--	(0.99,0.9999996)--	(1,1);

    \draw[densely dotted,black](0,1)--	(0.01,0.9313838)--	(0.02,0.8224342)--	(0.03,0.6966364)--	(0.04,0.5866296)--	(0.05,0.5029612)--	(0.06,0.4416951)--	(0.07,0.3964939)--	(0.08,0.3624161)--	(0.09,0.3361206)--	(0.1,0.3154035)--	(0.11,0.2987926)--	(0.12,0.285281)--	(0.13,0.2741615)--	(0.14,0.2649248)--	(0.15,0.2571959)--	(0.16,0.2506923)--	(0.17,0.245198)--	(0.18,0.2405445)--	(0.19,0.2365988)--	(0.2,0.2332546)--	(0.21,0.2304257)--	(0.22,0.2280419)--	(0.23,0.2260453)--	(0.24,0.2243877)--	(0.25,0.2230291)--	(0.26,0.2219359)--	(0.27,0.2210797)--	(0.28,0.2204367)--	(0.29,0.2199868)--	(0.3,0.2197131)--	(0.31,0.2196015)--	(0.32,0.2196402)--	(0.33,0.2198198)--	(0.34,0.2201325)--	(0.35,0.2205725)--	(0.36,0.2211355)--	(0.37,0.2218188)--	(0.38,0.2226211)--	(0.39,0.2235427)--	(0.4,0.2245853)--	(0.41,0.2257523)--	(0.42,0.2270486)--	(0.43,0.2284809)--	(0.44,0.2300579)--	(0.45,0.2317906)--	(0.46,0.2336924)--	(0.47,0.23578)--	(0.48,0.2380736)--	(0.49,0.2405979)--	(0.5,0.2433831)--	(0.51,0.2464663)--	(0.52,0.2498936)--	(0.53,0.2537222)--	(0.54,0.2580247)--	(0.55,0.2628935)--	(0.56,0.2684484)--	(0.57,0.2748469)--	(0.58,0.2822995)--	(0.59,0.2910928)--	(0.6,0.301624)--	(0.61,0.3144544)--	(0.62,0.3303907)--	(0.63,0.3506086)--	(0.64,0.3768276)--	(0.65,0.4115238)--	(0.66,0.4580563)--	(0.67,0.5202795)--	(0.68,0.6006668)--	(0.69,0.696066)--	(0.7,0.7937596)--	(0.71,0.8758926)--	(0.72,0.9321922)--	(0.73,0.964998)--	(0.74,0.9823219)--	(0.75,0.9910825)--	(0.76,0.9954699)--	(0.77,0.9976799)--	(0.78,0.9988048)--	(0.79,0.9993829)--	(0.8,0.999682)--	(0.81,0.9998371)--	(0.82,0.9999175)--	(0.83,0.9999588)--	(0.84,0.9999799)--	(0.85,0.9999904)--	(0.86,0.9999956)--	(0.87,0.999998)--	(0.88,0.9999992)--	(0.89,0.9999997)--	(0.9,0.9999999)--	(0.91,1)--	(0.92,1)--	(0.93,1)--	(0.94,1)--	(0.95,1)--	(0.96,1)--	(0.97,1)--	(0.98,1)--	(0.99,1)--	(1,1);

     \draw[dashdotted,orange] (0,1)--	(0.01,0.8773251)--	(0.02,0.6045782)--	(0.03,0.4433516)--	(0.04,0.362446)--	(0.05,0.3153774)--	(0.06,0.285186)--	(0.07,0.2645996)--	(0.08,0.2499853)--	(0.09,0.2393241)--	(0.1,0.2314022)--	(0.11,0.2254458)--	(0.12,0.2209378)--	(0.13,0.2175195)--	(0.14,0.2149348)--	(0.15,0.2129963)--	(0.16,0.2115641)--	(0.17,0.2105322)--	(0.18,0.2098196)--	(0.19,0.2093633)--	(0.2,0.2091147)--	(0.21,0.2090355)--	(0.22,0.2090961)--	(0.23,0.2092731)--	(0.24,0.2095485)--	(0.25,0.2099084)--	(0.26,0.2103423)--	(0.27,0.2108421)--	(0.28,0.2114023)--	(0.29,0.212019)--	(0.3,0.2126896)--	(0.31,0.2134133)--	(0.32,0.2141898)--	(0.33,0.2150201)--	(0.34,0.2159061)--	(0.35,0.2168502)--	(0.36,0.2178558)--	(0.37,0.218927)--	(0.38,0.2200687)--	(0.39,0.2212867)--	(0.4,0.2225875)--	(0.41,0.2239788)--	(0.42,0.2254695)--	(0.43,0.2270698)--	(0.44,0.2287917)--	(0.45,0.2306488)--	(0.46,0.2326574)--	(0.47,0.2348367)--	(0.48,0.2372094)--	(0.49,0.2398031)--	(0.5,0.2426512)--	(0.51,0.2457953)--	(0.52,0.2492874)--	(0.53,0.2531945)--	(0.54,0.2576044)--	(0.55,0.2626358)--	(0.56,0.2684555)--	(0.57,0.275308)--	(0.58,0.2835742)--	(0.59,0.2938961)--	(0.6,0.3074777)--	(0.61,0.3269357)--	(0.62,0.3592813)--	(0.63,0.4295247)--	(0.64,0.6426961)--	(0.65,0.9501217)--	(0.66,0.9966039)--	(0.67,0.9996207)--	(0.68,0.9999377)--	(0.69,0.9999869)--	(0.7,0.9999968)--	(0.71,0.9999991)--	(0.72,0.9999997)--	(0.73,0.9999999)--	(0.74,1)--	(0.75,1)--	(0.76,1)--	(0.77,1)--	(0.78,1)--	(0.79,1)--	(0.8,1)--	(0.81,1)--	(0.82,1)--	(0.83,1)--	(0.84,1)--	(0.85,1)--	(0.86,1)--	(0.87,1)--	(0.88,1)--	(0.89,1)--	(0.9,1)--	(0.91,1)--	(0.92,1)--	(0.93,1)--	(0.94,1)--	(0.95,1)--	(0.96,1)--	(0.97,1)--	(0.98,1)--	(0.99,1)--	(1,1);
     
      \draw[smooth,purple] (0,1)--	(0.01,0.8255251)--	(0.02,0.5089066)--	(0.03,0.3784658)--	(0.04,0.3155066)--	(0.05,0.2794492)--	(0.06,0.2567722)--	(0.07,0.241683)--	(0.08,0.2312803)--	(0.09,0.2239482)--	(0.1,0.2187146)--	(0.11,0.2149602)--	(0.12,0.2122723)--	(0.13,0.2103668)--	(0.14,0.2090428)--	(0.15,0.2081556)--	(0.16,0.2075993)--	(0.17,0.207296)--	(0.18,0.2071876)--	(0.19,0.2072309)--	(0.2,0.2073937)--	(0.21,0.2076521)--	(0.22,0.2079884)--	(0.23,0.2083898)--	(0.24,0.2088469)--	(0.25,0.2093534)--	(0.26,0.2099048)--	(0.27,0.2104987)--	(0.28,0.2111338)--	(0.29,0.2118097)--	(0.3,0.2125272)--	(0.31,0.2132876)--	(0.32,0.2140929)--	(0.33,0.2149457)--	(0.34,0.215849)--	(0.35,0.2168066)--	(0.36,0.2178226)--	(0.37,0.2189017)--	(0.38,0.2200495)--	(0.39,0.2212721)--	(0.4,0.2225764)--	(0.41,0.2239704)--	(0.42,0.2254631)--	(0.43,0.2270649)--	(0.44,0.2287879)--	(0.45,0.2306458)--	(0.46,0.2326551)--	(0.47,0.2348348)--	(0.48,0.2372078)--	(0.49,0.2398017)--	(0.5,0.2426499)--	(0.51,0.245794)--	(0.52,0.249286)--	(0.53,0.253193)--	(0.54,0.2576028)--	(0.55,0.2626344)--	(0.56,0.2684554)--	(0.57,0.2753128)--	(0.58,0.2835968)--	(0.59,0.2939825)--	(0.6,0.3078125)--	(0.61,0.3284095)--	(0.62,0.3677338)--	(0.63,0.5111726)--	(0.64,0.9736529)--	(0.65,0.9996681)--	(0.66,0.9999843)--	(0.67,0.9999985)--	(0.68,0.9999998)--	(0.69,1)--	(0.7,1)--	(0.71,1)--	(0.72,1)--	(0.73,1)--	(0.74,1)--	(0.75,1)--	(0.76,1)--	(0.77,1)--	(0.78,1)--	(0.79,1)--	(0.8,1)--	(0.81,1)--	(0.82,1)--	(0.83,1)--	(0.84,1)--	(0.85,1)--	(0.86,1)--	(0.87,1)--	(0.88,1)--	(0.89,1)--	(0.9,1)--	(0.91,1)--	(0.92,1)--	(0.93,1)--	(0.94,1)--	(0.95,1)--	(0.96,1)--	(0.97,1)--	(0.98,1)--	(0.99,1)--	(1,1);  
      
      \node[draw=none,,color=green] at (0.45,0.35) {\scriptsize{$\ell_1(\alpha)$}};
    \node[draw=none,,color=blue] at (0.9,0.5) {\scriptsize{$\ell_2(\alpha)$}};
    \node[draw=none,,color=red] at (0.85,0.65) {\scriptsize{$\ell_{5}(\alpha)$}};
    \node[draw=none,,color=black] at (0.75,0.78) {\scriptsize{$\ell_{10}(\alpha)$}};
    \node[draw=none,,color=orange] at (0.59,0.72) {\scriptsize{$\ell_{20}(\alpha)$}};
    \node[draw=none,,color=purple] at (0.58,0.87) {\scriptsize{$\ell_{25}(\alpha)$}};
    
        \draw[loosely dashed, domain=0:1] plot (\x, {1-\x});
            \node[draw=none,,color=black] at (0.31,0.75) {\scriptsize{$1-\alpha$}};
            
        \draw[smooth,dashed] (0.553,0)--(0.553,0.265);
        \node[draw=none,color=black, thick] at (0.67,0.1) {\scriptsize{$\alpha=1-\sqrt{\rho}$}};
        \draw [->] (0.62,0.08)--(0.56,0.01);
     
\end{tikzpicture}
\caption{The sequence of functions $\ell_n(\alpha):=L_{n-1}\left(\alpha(1-\alpha)^{n-1}\right)$ when the service-rate sequence is geometric with decay $1-\alpha$: $\mu_n=\alpha(1-\alpha)^{n-1}$. The system parameters are $\mu=1$ and Poisson arrivals with rate $\lambda=\rho=0.2$.}\label{fig:l_alpha}
\end{figure}
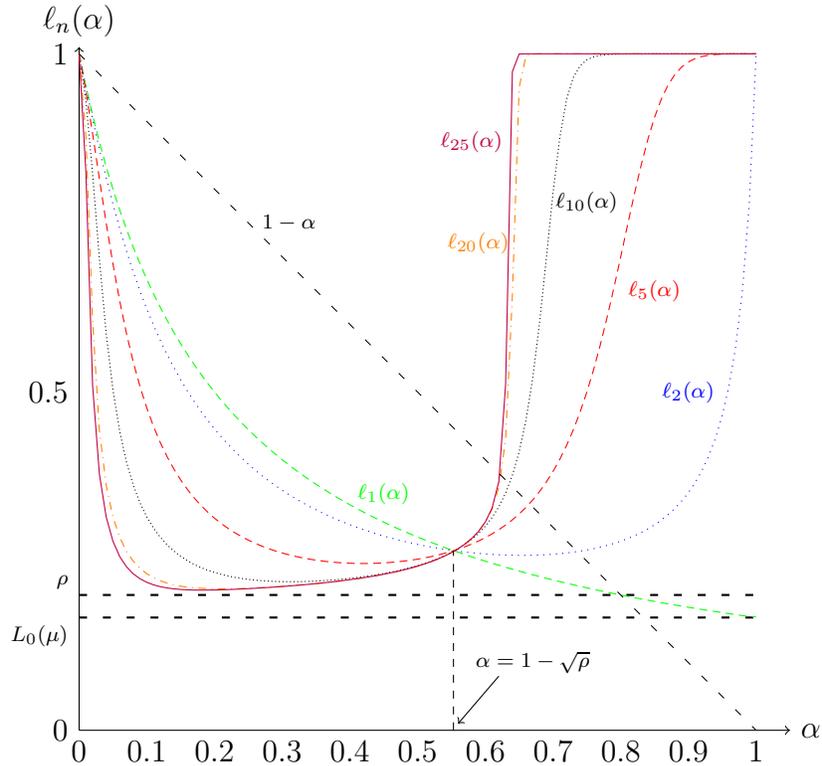

The invariance of the limit function $\ell(\alpha)$ also has implications on the delay-minimization problem: if $\ell(\alpha)=\ell$ for all $\alpha\in(0,\overline{\alpha})$ then the tail of the delay minimization objective function has a very simple form, $\frac{\ell^n}{(1-\alpha)^n}$, and the optimal $\alpha$ can be computed as described below. 

Suppose now that the blocking probability for all $n\geq 1$ is $p_n=\ell^n$, and consequently $q_n=(1-\ell)\ell^{n-1}$, where $\ell<1$. We already established in Lemma \ref{lemma:L_limit}b and Proposition \ref{prop:equivalence} that this is a reasonable approximation for the tail behaviour of the expected delay for any feasible service with finite delay. In the sequel (specifically in Proposition \ref{prop:l_invariance}) we will also show that if $\{\mu_n\}_{n=1}^\infty\in\mathcal{FD}$ then $\ell$ has a certain degree of invariance to the tail of $\{\mu_n\}_{n=1}^\infty$, thus providing additional justification for the use of approximation of the optimal solution with a fixed $\ell$. The optimal service-rate sequence for such a system is the solution to an infinite dimensional convex program on a simplex: 
\begin{equation}\label{CP:tail}
\begin{split}
\underset{\{\mu_n\}_{n=1}^\infty\in\mathcal{M}}{\text{min}} & \frac{1-\ell}{\ell}\sum_{n=1}^\infty \frac{\ell^n}{\mu_n} \\
\text{s.t.} & \left\lbrace \mu_n> 0, \ \forall n\geq 1, \  \sum_{n=1}^{\infty} \mu_n=1 \right\rbrace.
\end{split}
\end{equation}

We refer to \eqref{CP:tail} as the Tail Approximation Program (TAP). The following proposition asserts that the solution to the TAP is in $\mathcal{M}_g\cap\mathcal{M}$ with $\alpha=(1-\sqrt{\ell})$. This solution resembles the square-root optimal capacity allocation in a Jackson network (see \cite{book_K1976}, p. 329)\footnote{This observation was made by Johan van Leeuwaarden.}, but there is no direct link between the models. The program is a convex infinite horizon program, in the sense of \cite{G1983} (for general optimality conditions see \cite{book_BP2012}, p. 153), which allows us to find the optimal solution as a limit of finite dimensional programs. 

\begin{proposition}\label{prop:TAP}
The solution to \eqref{CP:tail} is $\mu_n=(1-\ell^\frac{1}{2})\ell^\frac{n-1}{2}$, $\forall n\geq 1$.
\end{proposition}
\begin{proof}
First of all we argue that the optimal service-rate sequence is non-increasing by applying a simple interchange argument. Suppose that $\{\mu_n\}_{n=1}^\infty$ is an optimal service-rate sequence such that $\mu_i<\mu_j$ for some $i<j$. The contribution of elements $i$ and $j$ to the objective function is
\[
\frac{\ell^i}{\mu_i}+\frac{\ell^j}{\mu_j}.
\]
If $i<j$ then $\ell^i>\ell^j$, which means that a greater weight is given to $\frac{1}{\mu_i}$ which is bigger than $\frac{1}{\mu_j}$. Hence, we can improve the objective without deviating from the capacity constraint by switching the values of $\mu_i$ and $\mu_j$. This contradicts the assumption that the sequence is optimal.

The objective function is an infinite sum of convex single-variable functions. We first consider the finite program for an integer $M$,
\[
\begin{split}
\underset{\{\mu_n\}_{n=1}^M\in\mathcal{M}}{\text{min}} & \frac{1-\ell}{\ell}\sum_{n=1}^M \frac{\ell^n}{\mu_n} \\
\text{s.t.} & \left\lbrace \mu_n\geq 0, \ \forall n\geq 1, \  \sum_{n=1}^{M} \mu_n=1 \right\rbrace.
\end{split}
\]
Every element of the objective function is unbounded as $\mu_n\to 0$ and therefore the solution is in the interior of the constraint set. This means that every element satisfies the first-order condition
\[
\frac{\ell^n}{\mu_n^2}=\kappa,\quad 1\leq n\leq M
\]
where $\kappa$ is the Lagrange multiplier for the equality constraint
\[
-\kappa\left(\sum_{n=1}^{M} \mu_n-1\right)=0.
\]
Simple algebra then yields
\[
\mu_n=\frac{1}{\sqrt{\kappa}}\ell^\frac{n}{2},
\]
and by applying the capacity constraint,
\[
\sum_{n=1}^M \frac{1}{\sqrt{\kappa}}\ell^\frac{n}{2}=1,
\]
we derive that $\sqrt{\kappa}=\frac{\ell^\frac{1}{2}\big(1-\ell^\frac{M}{2}\big)}{1-\ell^\frac{1}{2}}$.
Finally, by taking $M\to\infty$ we conclude that the optimal solution to \eqref{CP:tail} is $\mu_n=(1-\ell^\frac{1}{2})\ell^\frac{n-1}{2}$. 
\end{proof}

\section{Optimization and approximation}\label{sec:opt}
We are interested in solving the mathematical program,
\begin{equation}\label{MP}
\min_{\{\mu_n\}_{n=1}^\infty\in\mathcal{M}} \sum_{n=1}^\infty \frac{q_n}{\mu_n}.
\end{equation}

This program can be formulated as an infinite horizon Markov Decision Process with state and action dependent discount factor (see \cite{WG2011}). The idea is that at every step $n$ we consider a new system with inter-arrival distribution given by the overflows from server $n-1$ and the remaining capacity constraint. The discount factor at step $n$ will be given by $L_{n-1}(\mu_n)$, the blocking probability when a customer overflows to server $n$.

First we define the mapping
\[
\hat{L}(x,L)(s)=\frac{L(s+x)}{1-L(s)+L(s+x)}:\mathbb{R}\times\mathcal{L}\to\mathcal{L},
\]
where $\mathcal{L}$ is the space of non-increasing functions from $[0,\infty)$ to $[0,1]$. For any $n\geq 1$, given the overflow distribution $L_{n-1}$ we take advantage of the recursive form of $q_n$ in \eqref{eq:qn_recursion} to obtain
\[
q_{n+1}=\left(1-\hat{L}(\mu_n,L_{n-1})(\mu_{n+1})\right)p_{n},
\]
where by \eqref{eq:blocking},
\[
p_{n}=L_{n-1}(\mu_n)p_{n-1}.
\]
The objective function of \eqref{MP} can then be written as
\[
\begin{split}
& (1-L_0(\mu_1))\frac{1}{\mu_1}+L_0(\mu_1)\left(1-\hat{L}(\mu_{1},L_0)(\mu_2)\right)\frac{1}{\mu_2} \\
& +L_0(\mu_1)\hat{L}(\mu_{1},L_0)(\mu_2)\left(1-\hat{L}(\mu_{2},\hat{L}(\mu_{1},L_0)(\mu_2))(\mu_3)\right)\frac{1}{\mu_3}+\ldots\quad .
\end{split}
\]
Therefore, an equivalent program to \eqref{MP} is given by the Bellman equation
\begin{equation}\label{DP2}
v(\mu,L)=\min_{\{x\in[0,\mu]\}}\left\lbrace\big(1-L(x)\big)\frac{1}{x}+L(x)v\big(\mu-x,\hat{L}(x,L)\big)\right\rbrace,
\end{equation}
with the objective $v(\mu,L_0)$. In every step $\mu$ is the total available capacity and $L$ is the LST defining the external arrival process to the system. While \eqref{DP2} has an elegant form it is not straightforward to solve even numerically. This is due to the infinite dimensional state space $\mathcal{L}$, which is a space of continuous functions. We next suggest an equivalent program with a simpler state space that includes the server index and the capacities that have been allocated.

For any given exogenous arrival distribution $L_0$ we can compute the values of $L_n(s)$ given the sequence $\{\mu_1,\ldots,\mu_{n-1}\}$ using the recursive formula \eqref{eq:LST}. The program \eqref{DP2} with capacity constraint $\mu$ can then be defined by the Bellman equation
\begin{equation}\label{DP}
v_n(\mu_1,\ldots,\mu_{n-1})=\min_{\mu_n\leq \mu-s_{n-1}}\left\lbrace\frac{q_n}{\mu_n}+v_{n+1}(\mu_1,\ldots,\mu_n)\right\rbrace,\quad n\geq 1,
\end{equation}
where $s_n:=\sum_{i=1}^{n}\mu_i$. The overall objective is $v_1(\emptyset)$.

Unfortunately there is an additional problem of computational complexity. Specifically, computing $q_n$ requires computing the recursion for $L_{n-1}(\mu_n)$ which is of the magnitude of $2^n$ steps. In the sequel we propose a numerical approximation method that relies on the solution of \eqref{DP} for a small number of steps with the TAP solution \eqref{CP:tail} as an initial condition.

Observe that there is no direct restriction for the solution of \eqref{MP} to be non-increasing, which is necessary if customers always go to the fastest server available. We next argue that the optimal sequence is indeed non-increasing, even without the explicit constraint. In Proposition \ref{prop:TAP} the explicit geometric decay rate of the optimal service sequence was shown to be $\sqrt{\ell}$ for the approximation model, whereas in the general case we only know that it decreases but not at what rate.

\begin{lemma}\label{lemma:DP_decreasing}
The solution $\{\mu_n\}_{n=1}^\infty$ of \eqref{MP} is a non-increasing sequence. 
\end{lemma}
\begin{proof}
Suppose that $\{\mu_n\}_{n=1}^\infty$ is an optimal solution such that $\mu_n<\mu_{n+1}$ for some $n\geq 1$. The average expected delay is
\[
\E S = \E(S|Y<n)\P(Y<n)+\E(S|Y\geq n)\P(Y\geq n).
\]
If the rates of server $n$ and $n+1$ are interchanged then first summand is unchanged, while the second is decreased because all blocking probabilities $p_k$ for $k \geq n$ decrease (see \cite{NE1981}), thus, contradicting the optimality of the sequence. 
\end{proof}

A nice property of decreasing service-rate sequence is given to us by Lemma \ref{lemma:LST_properties}c, which states that the sequence of blocking properties $p_n$ is discrete convex:
\[
p_{n+1}+p_{n-1}>2p_n,\quad n\geq 2.
\] 
Recall that $p_n=\prod_{i=1}^nL_{i-1}(\mu_i)$, so in terms of the LST sequence this is equivalent to
\[
L_{n-1}(\mu_n)(1-L_{n}(\mu_{n+1}))<1-L_{n-1}(\mu_n),\quad n\geq 2,
\]
and thus
\[
q_{n+1}=(1-L_{n}(\mu_{n+1}))L_{n-1}(\mu_n)\prod_{i=1}^{n-1}L_{i-1}(\mu_i)<(1-L_{n-1}(\mu_n))\prod_{i=1}^{n-1}L_{i-1}(\mu_i)=q_n.
\]
This means that the sequence $q_{n}$ is decreasing, hence the weights of the increasing sequence of expected service times, $\frac{1}{\mu_n}$, in the objective function of \eqref{MP} is decreasing. 

\subsection{Approximate solution}\label{sec:opt_approx}
If $\ell<1$ then \eqref{eq:blocking} gives us a geometric approximation of the tail behaviour of the blocking probabilities $p_n\approx \ell^{n}$. Thus, for large $M$ we set
\[
q_n= p_M\ell^{n-(M+1)}(1-\ell),\ \forall n>M.
\]
Due to Proposition \ref{prop:TAP} we have that if the sequence $\ell_n$ does not vary by much then a good approximation for the optimal solution is given by a geometric service-rate sequence with decay rate $\sqrt{\ell}$. Specifically, the approximately optimal tail series is
\begin{equation}\label{eq:mu_tail}
\mu_n=\mu_{n-1}\sqrt{\ell}=\mu_M\ell^{\frac{n-M}{2}},\quad \forall n>M,
\end{equation}
and the approximate optimal residual is
\[
\sum_{n=M+1}^\infty \frac{q_n}{\mu_n}\approx \frac{p_M(1-\ell)}{\mu_M\ell}\sum_{n=1}^\infty \left(\frac{\ell}{\sqrt{\ell}}\right)^{n}=\frac{p_M(1-\ell)}{\mu_M\sqrt{\ell}(1-\sqrt{\ell})}=\frac{p_M(1+\sqrt{\ell})}{\mu_M\sqrt{\ell}}.
\]

For small values of $M$ ($\leq 25$) we can accurately compute $q_1,\ldots,q_M$ and approximate the optimal residual by using the TAP solution \eqref{CP:tail}. This yields the approximation
\[
\E S \approx \sum_{n=1}^M \frac{q_n}{\mu_n}+r_M,
\]
where
\[
r_M:=\frac{p_M(1+\sqrt{\ell_M})}{\mu_M\sqrt{\ell_M}},
\]
and $\ell_M:=L_{M-1}(\mu_M)$. The term $r_M$ represents the residual of the expected delay given by the tail approximation.

According to Proposition \ref{prop:equivalence}b, $\ell<1$ for any sequence with finite delay. A finite-horizon dynamic program that approximates \eqref{DP} can now be formulated: for $1\leq n\leq M$,
\begin{equation}\label{DP_approx}
v_n^{(M)}(\mu_1,\ldots,\mu_{n-1})=\min_{\mu_n\leq \mu-s_{n-1}}\left\lbrace\frac{q_n}{\mu_n}+v_{n+1}^{(M)}(\mu_1,\ldots,\mu_n)\right\rbrace,
\end{equation}
with initial condition $v^{(M)}_{M+1}(\mu_1,\ldots,\mu_M)=r_M$ and the objective $v_1^{(M)}(\mu)$. We can increase $M$ until $r_M$ is lower than some tolerance parameter, or alternatively until the change in the $L_{n-1}(\mu_n)$ sequence is smaller than some parameter.

The tail sequence $\{\mu_n\}_{n=M+1}^\infty$ needs to satisfy the capacity constraint,
\[
\sum_{n=M}^\infty \mu_m\leq \mu-s_{M-1} ,
\]
which according to \eqref{eq:mu_tail} yields
\begin{equation}\label{eq:opt_muM}
\mu_M\leq (1-\sqrt{\ell_M})(\mu-s_{M-1}).
\end{equation}

In an optimal allocation \eqref{eq:opt_muM} will have an equality, as there is no gain from not allocating all of the capacity. Lemma \ref{lemma:LST_properties}a implies that for every $(\mu_1,\ldots,\mu_{M-1})$ there is a unique $\mu_M$ for which an equality holds. Therefore, the dynamic program effectively only has $M-1$ steps. 

The implementation of the approximating dynamic program is obtained using a standard fixed point search algorithm. Let $\mu_0:=0$, and
\[
\mu_n^*(\mu_1,\ldots,\mu_{n-1})=\left\lbrace\begin{array}{cc}
\argmin_{\mu_1\leq \mu}\left\lbrace\frac{q_1}{\mu_1}+v_{2}^{(M)}(\mu_1)\right\rbrace, & n=1 \\
\argmin_{\mu_n\leq \mu-s_{n-1}}\left\lbrace\frac{q_n}{\mu_n}+v_{n+1}^{(M)}(\mu_1,\ldots,\mu_n)\right\rbrace, & 2 \leq  n<M, \\
(1-\sqrt{\ell_M})(\mu-s_{M-1}), & n=M,
\end{array}\right.
\]
for $n=1,\ldots,M$. Then $(\mu_1,\ldots,\mu_M)$ is an optimal solution if $\mu_n^*(\mu_1,\ldots,\mu_{n-1})=\mu_n$, for all $1\leq n\leq M$. Note that while the constraint at step $n$ depends only on the sum $s_{n-1}=\sum_{i=1}^n\mu_i$, the value function $v_{n}^{(M)}$ depends on the entire vector $(\mu_1,\ldots,\mu_{n-1})$ through the overflow distribution $L_n$. A simplified description of such an algorithm is as follows: 
\begin{enumerate}[label=(\arabic*)]
\item For $n=M,M-1,\ldots,1$  compute $\mu_n^*(\mu_1,\ldots,\mu_{n-1})$ for any allocation $(\mu_1\ldots,\mu_{n-1})$ and the corresponding value $v^{(M)}_{n}(\mu_1,\ldots,\mu_n^*(\mu_1,\ldots,\mu_{n-1}))$.
\item The vector
\[
 \left(\mu_1^*,\mu_2^*\left(\mu_1^*\right),\ldots,\mu_{M}^*\left(\mu_1^*,\ldots,\mu_{M-1}^*\left(\mu_1^*,\mu_2^*\left(\mu_1^*\right),\ldots\right)\right)\right),
\]
satisfies the fixed point condition and is an optimal solution.
\end{enumerate}
The most naive and exhaustive way to solve the approximate program is to compute the values for all possible $M$-dimensional allocations on a discrete grid with increments of size $\delta>0$. This is of course computationally expensive, in the magnitude of $\left(\frac{\mu}{\delta}\right)^M$. The search at any stage can be carried out in a much more efficient manner, such as bisection, and then the functions do not have to be evaluated at every point in the continuous search interval for every $\mu_n\in(0,\mu-s_{n-1})$. In practice, the search finds the optimal value in a very small number of computations using various generic optimization packages and the computational bottleneck is the evaluation of $L_n()$ for increasing $n$. 

In the TAP the blocking probability is assumed to decay at a constant rate, specifically the limit $\ell$. For this to be a good approximation the tail of the sequence $\ell_n:=L_{n-1}(\mu_n)$ needs to be somehow insensitive to changes in the tail of the service-rate sequence. This behaviour appeared in Observation 2 in Section \ref{sec:geometric} and was illustrated in Figure \ref{fig:l_alpha}. To strengthen the justification of this approximation we further show that the tail of any service-rate sequence with finite delay is decreasing and is bounded from below by an increasing sequence. 

\begin{proposition}\label{prop:l_invariance}
If $\{\mu_n\}_{n=1}^\infty\in\mathcal{FD}$ then the sequence $\{\ell_n\}$ has a decreasing tail, and is bounded from below by an increasing sequence $\{\underline{\ell}_n\}$.
\end{proposition}
\begin{proof}
According to Lemma \ref{lemma:L_limit} the LST sequence has a lower bound of $L_0(\mu)$, i.e.\ the external input LST with all of the capacity. This argument can be repeated for every overflow distribution $L_{n-1}$ into server $n$, when the remaining capacity is $\mu-s_{n-1}$. And so at any step $n$ of the dynamic program \eqref{DP_approx} we have
\[
\ell_n= L_{n-1}(\mu_n)\geq L_{n-1}(\mu-s_{n-1})=:\underline{\ell}_n.
\]
By \eqref{eq:LST},
\[
L_n(\mu-s_n)=\frac{L_{n-1}(\mu_n+\mu-s_n)}{1-L_{n-1}(\mu-s_n)+L_{n-1}(\mu_n+\mu-s_n)},
\]
and, as $\mu_n+\mu-s_n=\mu-s_{n-1}$, and the denominator is smaller than one,
\[
\underline{\ell}_{n+1}=L_n(\mu-s_n)> L_{n-1}(\mu-s_{n-1})=\underline{\ell}_{n}.
\]
Hence the sequence of lower bounds, $\underline{\ell}_n$, is increasing with $n$. Furthermore, if the sequence $\{\mu_n\}_{n=1}^\infty$ is non-increasing and has finite delay then the sequence $\frac{q_n}{\mu_n}$ converges to zero and is therefore decreasing at the tail. Using \eqref{eq:qn_recursion}, this implies that for large $n$:
\[
\frac{\mu_{n+1}(1-\ell_n)p_{n-1}}{\mu_n(1-\ell_{n+1})p_n}<1,
\]
which yields
\[
\frac{1-\ell_n}{1-\ell_{n+1}}<\frac{\mu_n}{\mu_{n+1}}\ell_n<1.
\]
The last inequality comes from the finite delay condition in Proposition \ref{prop:equivalence}c that demands that the decay of the service-rate sequence be slower than that of the blocking probabilities. We therefore conclude that $\ell_n>\ell_{n+1}$ at the tail. 
\end{proof}

To summarize, for any reasonable service-rate sequence, that is non-increasing and with finite delay, the sequence $\ell_n$ is decreasing and is also bounded from below by an increasing sequence. This shows that changing the service-rate sequence at the tail has a small, or bounded, effect on the limit $\ell$, as long as the finite delay condition is met. In Figure \ref{fig:Ld} the lower bound sequence is illustrated alongside the LST sequence for the approximat optimal solution for an example set of parameters. Indeed, $\ell_n$ approaches the lower bound sequence $\underline{\ell}_n$ very quickly. In this example we have that $\frac{\ell_{15}-\underline{\ell}_{15}}{\ell_{15}}\eqsim 0.025$, which indicates that the error term is very accurate even for $M=15$ (and at $M=19$ the normalized error is already smaller than $0.001$). For other examples similar behaviour was observed, and, as expected, for higher levels of $\rho$ a bigger $M$ is required to achieve good accuracy.

\begin{figure}[H]
\centering
\begin{tikzpicture}[xscale=0.6,yscale=7]
  \def\xmin{0.8}
  \def\xmax{16}
  \def\ymin{0.15}
  \def\ymax{0.65}
    \draw[->] (\xmin,\ymin) -- (\xmax,\ymin) node[right] {$n$} ;
    \draw[->] (\xmin,\ymin) -- (\xmin,\ymax); 
    \foreach \x in {1,2,3,4,5,6,7,8,9,10,11,12,13,14,15}
    \node at (\x,\ymin) [below] {\x};
    \foreach \y in {0.2,0.3,0.4,0.5,0.6}
    \node at (\xmin,\y) [left] {\y};
	
    \draw[red,densely dashed] (1,0.2857)--	(2,0.3103)--	(3,0.3338)--	(4,0.3557)--	(5,0.3759)--	(6,0.3942)--	(7,0.4105)--	(8,0.4247)--	(9,0.4369)--	(10,0.4473)--	(11,0.456)--	(12,0.4631)--	(13,0.469)--	(14,0.4737)--	(15,0.4774);
    
    \draw[blue] (1,0.5741)--	(2,0.563)--	(3,0.5528)--	(4,0.5435)--	(5,0.535)--	(6,0.5273)--	(7,0.5204)--	(8,0.5142)--	(9,0.5089)--	(10,0.5042)--	(11,0.5003)--	(12,0.497)--	(13,0.4942)--	(14,0.492)--	(15,0.4902);  
    
    \node[draw=none,color=red] at (11,0.38) {$\underline{\ell}_n$};
    \node[draw=none,color=blue] at (11,0.55) {$\ell_n$};
    
\end{tikzpicture}
 \caption{LST sequence (of the approximate optimal service-rate sequence) and the lower bound sequence. Example parameters: Total capacity of $\mu=1$ and Poisson arrivals with rate $\lambda=0.4$.}\label{fig:Ld}
\end{figure}
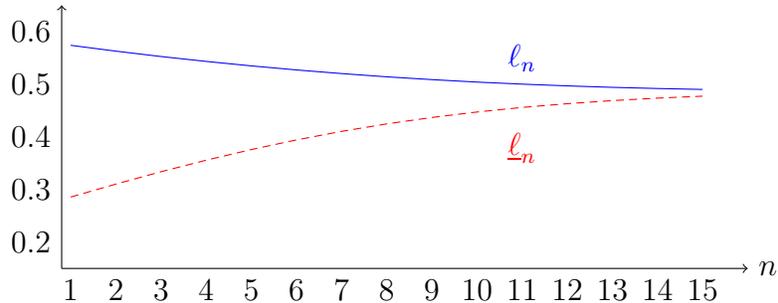

\section{Numerical analysis}\label{sec:numerical}
In this section we assume $\mu=1$ and that the external inter-arrival distribution is Gamma$(k,k\lambda)$. In this case the overall utilization is $\rho=\frac{k\lambda}{k}=\lambda$ and the variance of the exogenous inter-arrival times is $\frac{1}{k\lambda^2}=\frac{1}{k\rho^2}$. We can therefore examine different levels of utilization and variance by changing $\rho$ and $k$. For high levels of $\rho$ this analysis is quite general as the blocking probabilities in the heavy-traffic approximation of the GI/M/$c$ system with many servers are known to depend only on the first two moments of the arrival distribution (see \cite{W1984}).

Figure \ref{fig:opt} illustrates the approximate optimal service-rate sequence by solving \eqref{DP_approx} for different parameter values with $M=15$. The first thing to observe is that in all examples the approximate optimal service-rate sequence is very close to geometric. 

In Table \ref{tbl:opt} we present the approximate optimal values of $\E S$ and the respective tail approximations $r_M$. For high levels of $\rho$ the contribution of the tail approximation is substantial and hence potentially less accurate. However, we find that the sequence of $L_{n-1}(\mu_n)$ stabilizes very quickly and therefore the approximation $\ell\eqsim L_{M-1}(\mu_M)$ is quite accurate, suggesting the error terms provide a decent approximation. The sequence of LST of the approximate optimal service-rate sequence are illustrated in Figure \ref{fig:lst}. In all examples the sequence indeed stabilizes very fast, hence the tail approximation using the value of $\ell_M$ is appropriate. This stability result was further verified by running computations for higher values, exact up to $M=25$ and based on a discrete event simulation of the system for $M>25$, for a large number of servers using the approximate optimal service-rate sequence. In particular, the LST sequences, for example those displayed in Figure 5, remain almost constant when taking much larger $n$ than $10$.

\begin{figure}[H]
\centering
\begin{subfigure}{.48\linewidth}
\begin{tikzpicture}[xscale=0.6,yscale=2.5]
  \def\xmin{0.8}
  \def\xmax{10}
  \def\ymin{0}
  \def\ymax{0.9}
    \draw[->] (\xmin,\ymin) -- (\xmax,\ymin) node[right] {$n$} ;
    \draw[->] (\xmin,\ymin) -- (\xmin,\ymax) node[above] {$\mu_n$} ;
    \foreach \x in {1,2,3,4,5,6,7,8,9,10}
    \node at (\x,\ymin) [below] {\x};
    \foreach \y in {0,0.2,0.4,0.6,0.8}
    \node at (\xmin,\y) [left] {\y};
    
    \draw[smooth,green] (1,0.5157474146)--	(2,0.2497520189)--	(3,0.1209430609)--	(4,0.0585669899)--	(5,0.0283612163)--	(6,0.0137339923)--	(7,0.0066507213)--	(8,0.003220629)--	(9,0.0015595979)--	(10,0.0007552393);

    \foreach \Point in {(1,0.5157474146),	(2,0.2497520189),	(3,0.1209430609),	(4,0.0585669899),	(5,0.0283612163),	(6,0.0137339923),	(7,0.0066507213),	(8,0.003220629),	(9,0.0015595979),	(10,0.0007552393)}	
    {\node[green] at \Point {\scalebox{0.3}{$\triangle$}};}
    
    \draw[dotted,blue] (1,0.6587197696618)--	(2,0.2248080347185)--	(3,0.0767225378706)--	(4,0.0261838853966)--	(5,0.0089360424393)--	(6,0.003049694622)--	(7,0.0010408004831)--	(8,0.0003552046286)--	(9,0.0001212243175)--	(10,0.000041371463);

    \foreach \Point in {(1,0.6587197696618),	(2,0.2248080347185),	(3,0.0767225378706),	(4,0.0261838853966),	(5,0.0089360424393),	(6,0.003049694622),	(7,0.0010408004831),	(8,0.0003552046286),	(9,0.0001212243175),	(10,0.000041371463)}
    {\node[blue] at \Point {\scalebox{0.3}{\textbullet}};}
    
      \draw[densely dotted,red] ( 1 , 0.35825 )-- ( 2 , 0.22991 )-- ( 3 , 0.14754 )-- ( 4 , 0.09468 )-- ( 5 , 0.06076 )-- ( 6 , 0.03899 )-- ( 7 , 0.02502 )-- ( 8 , 0.01606 )-- ( 9 , 0.01031 )-- ( 10 , 0.00661 );

    \foreach \Point in {( 1 , 0.35825 ), ( 2 , 0.22991 ), ( 3 , 0.14754 ), ( 4 , 0.09468 ), ( 5 , 0.06076 ), ( 6 , 0.03899 ), ( 7 , 0.02502 ), ( 8 , 0.01606 ), ( 9 , 0.01031 ), ( 10 , 0.00661 )}	
    {\node[red] at \Point {\scalebox{0.3}{$\star$}};}
    
    \draw[dashed,black](1,0.839907672099729)--	(2,0.134462774447742)--	(3,0.0215264585772682)--	(4,0.00344622086508361)--	(5,0.00055171352074972)--	(6,0.00000883251018708771)--	(7,0.00000141401711705373)--	(8,0.00000022637329195996)--	(9,0.00000003624062728432)--	(10,0.00000000580184638651);

    \foreach \Point in {(1,0.839907672099729),	(2,0.134462774447742),	(3,0.0215264585772682),	(4,0.00344622086508361),	(5,0.00055171352074972),	(6,0.00000883251018708771),	(7,0.00000141401711705373),	(8,0.00000022637329195996),	(9,0.00000003624062728432),	(10,0.00000000580184638651)}	
    {\node[black] at \Point {\scalebox{0.3}{$\square$}};}
\end{tikzpicture}
\caption{$\rho=\lambda=0.2$}\label{fig:opt_a}
\end{subfigure}
\begin{subfigure}{.48\linewidth}
\begin{tikzpicture}[xscale=0.6,yscale=3.75]
  \def\xmin{0.8}
  \def\xmax{10}
  \def\ymin{0}
  \def\ymax{0.61}
    \draw[->] (\xmin,\ymin) -- (\xmax,\ymin) node[right] {$n$} ;
    \draw[->] (\xmin,\ymin) -- (\xmin,\ymax) node[above] {$\mu_n$} ;
    \foreach \x in {1,2,3,4,5,6,7,8,9,10}
    \node at (\x,\ymin) [below] {\x};
    \foreach \y in {0,0.2,0.4,0.6}
    \node at (\xmin,\y) [left] {\y};
    
    \draw[smooth,green] (1,0.3055711)--	(2,0.2121974)--	(3,0.147356)--	(4,0.1023283)--	(5,0.0710597)--	(6,0.0493459)--	(7,0.0342672)--	(8,0.0237961)--	(9,0.0165247)--	(10,0.0114753);

    \foreach \Point in {(1,0.3055711),	(2,0.2121974),	(3,0.147356),	(4,0.1023283),	(5,0.0710597),	(6,0.0493459),	(7,0.0342672),	(8,0.0237961),	(9,0.0165247),	(10,0.0114753)}	
    {\node[green] at \Point {\scalebox{0.3}{$\triangle$}};}
    
    \draw[dotted,blue] (1,0.4023266)--	(2,0.24045991)--	(3,0.14371649)--	(4,0.08589552)--	(5,0.05133747)--	(6,0.03068304)--	(7,0.01833844)--	(8,0.0109604)--	(9,0.00655074)--	(10,0.0039152);

    \foreach \Point in {(1,0.4023266),	(2,0.24045991),	(3,0.14371649),	(4,0.08589552),	(5,0.05133747),	(6,0.03068304),	(7,0.01833844),	(8,0.0109604),	(9,0.00655074),	(10,0.0039152)}
    {\node[blue] at \Point {\scalebox{0.3}{\textbullet}};}
    
      \draw[densely dotted,red] ( 1 , 0.20515 )-- ( 2 , 0.16306 )-- ( 3 , 0.12961 )-- ( 4 , 0.10302 )-- ( 5 , 0.08189 )-- ( 6 , 0.06509 )-- ( 7 , 0.05173 )-- ( 8 , 0.04112 )-- ( 9 , 0.03269 )-- ( 10 , 0.02598 );

    \foreach \Point in {( 1 , 0.20515 ), ( 2 , 0.16306 ), ( 3 , 0.12961 ), ( 4 , 0.10302 ), ( 5 , 0.08189 ), ( 6 , 0.06509 ), ( 7 , 0.05173 ), ( 8 , 0.04112 ), ( 9 , 0.03269 ), ( 10 , 0.02598 )}	
    {\node[red] at \Point {\scalebox{0.3}{$\star$}};}
    
    \draw[dashed,black] (1,0.54123667)--	(2,0.248299537)--	(3,0.1139107225)--	(4,0.0522580624)--	(5,0.0239740827)--	(6,0.01099843)--	(7,0.0050456764)--	(8,0.0023147713)--	(9,0.0010619322)--	(10,0.0004871755);

    \foreach \Point in {(1,0.54123667),	(2,0.248299537),	(3,0.1139107225),	(4,0.0522580624),	(5,0.0239740827),	(6,0.01099843),	(7,0.0050456764),	(8,0.0023147713),	(9,0.0010619322),	(10,0.0004871755)}	
    {\node[black] at \Point {\scalebox{0.3}{$\square$}};}
\end{tikzpicture}
\caption{$\rho=\lambda=0.4$}\label{fig:opt_b}
\end{subfigure}

\begin{subfigure}{.48\linewidth}
\begin{tikzpicture}[xscale=0.6,yscale=5.66]
  \def\xmin{0.8}
  \def\xmax{10}
  \def\ymin{0}
  \def\ymax{0.41}
    \draw[->] (\xmin,\ymin) -- (\xmax,\ymin) node[right] {$n$} ;
    \draw[->] (\xmin,\ymin) -- (\xmin,\ymax) node[above] {$\mu_n$} ;
    \foreach \x in {1,2,3,4,5,6,7,8,9,10}
    \node at (\x,\ymin) [below] {\x};
    \foreach \y in {0,0.1,0.2,0.3,0.4}
    \node at (\xmin,\y) [left] {\y};
    
    \draw[smooth,green] (1,0.159331)--	(2,0.133944)--	(3,0.112603)--	(4,0.094662)--	(5,0.079579)--	(6,0.0669)--	(7,0.056241)--	(8,0.04728)--	(9,0.039747)--	(10,0.033414);

    \foreach \Point in {(1,0.159331),	(2,0.133944),	(3,0.112603),	(4,0.094662),	(5,0.079579),	(6,0.0669),	(7,0.056241),	(8,0.04728),	(9,0.039747),	(10,0.033414)}	
    {\node[green] at \Point {\scalebox{0.3}{$\triangle$}};}
    
    \draw[dotted,blue] (1,0.212008)--	(2,0.167061)--	(3,0.131642)--	(4,0.103733)--	(5,0.081741)--	(6,0.064411)--	(7,0.050755)--	(8,0.039995)--	(9,0.031516)--	(10,0.024834);

    \foreach \Point in {(1,0.212008),	(2,0.167061),	(3,0.131642),	(4,0.103733),	(5,0.081741),	(6,0.064411),	(7,0.050755),	(8,0.039995),	(9,0.031516),	(10,0.024834)}
    {\node[blue] at \Point {\scalebox{0.3}{\textbullet}};}
    
      \draw[densely dotted,red] ( 1 , 0.10811 )-- ( 2 , 0.09642 )-- ( 3 , 0.086 )-- ( 4 , 0.0767 )-- ( 5 , 0.06841 )-- ( 6 , 0.06101 )-- ( 7 , 0.05442 )-- ( 8 , 0.04853 )-- ( 9 , 0.04329 )-- ( 10 , 0.03861 );

    \foreach \Point in {( 1 , 0.10811 ), ( 2 , 0.09642 ), ( 3 , 0.086 ), ( 4 , 0.0767 ), ( 5 , 0.06841 ), ( 6 , 0.06101 ), ( 7 , 0.05442 ), ( 8 , 0.04853 ), ( 9 , 0.04329 ), ( 10 , 0.03861 )}	
    {\node[red] at \Point {\scalebox{0.3}{$\star$}};}
    
    \draw[dashed,black] (1,0.289094)--	(2,0.2055187)--	(3,0.1461044)--	(4,0.1038665)--	(5,0.0738393)--	(6,0.0524928)--	(7,0.0373175)--	(8,0.0265292)--	(9,0.0188598)--	(10,0.0134075);

    \foreach \Point in {(1,0.289094),	(2,0.2055187),	(3,0.1461044),	(4,0.1038665),	(5,0.0738393),	(6,0.0524928),	(7,0.0373175),	(8,0.0265292),	(9,0.0188598),	(10,0.0134075)}	
    {\node[black] at \Point {\scalebox{0.3}{$\square$}};}
\end{tikzpicture}
\caption{$\rho=\lambda=0.6$}\label{fig:opt_c}
\end{subfigure}
\begin{subfigure}{.48\linewidth}
\begin{tikzpicture}[xscale=0.6,yscale=15]
  \def\xmin{0.8}
  \def\xmax{10}
  \def\ymin{0}
  \def\ymax{0.16}
    \draw[->] (\xmin,\ymin) -- (\xmax,\ymin) node[right] {$n$} ;
    \draw[->] (\xmin,\ymin) -- (\xmin,\ymax) node[above] {$\mu_n$} ;
    \foreach \x in {1,2,3,4,5,6,7,8,9,10}
    \node at (\x,\ymin) [below] {\x};
    \foreach \y in {0,0.05,0.1,0.15}
    \node at (\xmin,\y) [left] {\y};
    
    \draw[smooth,green] (1,0.06693)--	(2,0.06245)--	(3,0.05827)--	(4,0.05437)--	(5,0.05073)--	(6,0.04734)--	(7,0.04417)--	(8,0.04121)--	(9,0.03845)--	(10,0.03588);

    \foreach \Point in {(1,0.06693),	(2,0.06245),	(3,0.05827),	(4,0.05437),	(5,0.05073),	(6,0.04734),	(7,0.04417),	(8,0.04121),	(9,0.03845),	(10,0.03588)}	
    {\node[green] at \Point {\scalebox{0.3}{$\triangle$}};}
    
    \draw[dotted,blue] (1,0.08488)--	(2,0.07768)--	(3,0.07108)--	(4,0.06505)--	(5,0.05953)--	(6,0.05448)--	(7,0.04985)--	(8,0.04562)--	(9,0.04175)--	(10,0.0382);

    \foreach \Point in {(1,0.08488),	(2,0.07768),	(3,0.07108),	(4,0.06505),	(5,0.05953),	(6,0.05448),	(7,0.04985),	(8,0.04562),	(9,0.04175),	(10,0.0382)}
    {\node[blue] at \Point {\scalebox{0.3}{\textbullet}};}
    
      \draw[densely dotted,red] ( 1 , 0.04908 )-- ( 2 , 0.04667 )-- ( 3 , 0.04438 )-- ( 4 , 0.0422 )-- ( 5 , 0.04013 )-- ( 6 , 0.03816 )-- ( 7 , 0.03629 )-- ( 8 , 0.03451 )-- ( 9 , 0.03281 )-- ( 10 , 0.0312 );

    \foreach \Point in {( 1 , 0.04908 ), ( 2 , 0.04667 ), ( 3 , 0.04438 ), ( 4 , 0.0422 ), ( 5 , 0.04013 ), ( 6 , 0.03816 ), ( 7 , 0.03629 ), ( 8 , 0.03451 ), ( 9 , 0.03281 ), ( 10 , 0.0312 )}	
    {\node[red] at \Point {\scalebox{0.3}{$\star$}};}
    
    \draw[dashed,black] (1,0.11064)--	(2,0.0984)--	(3,0.08751)--	(4,0.07783)--	(5,0.06922)--	(6,0.06156)--	(7,0.05475)--	(8,0.04869)--	(9,0.0433)--	(10,0.03851);

    \foreach \Point in {(1,0.11064),	(2,0.0984),	(3,0.08751),	(4,0.07783),	(5,0.06922),	(6,0.06156),	(7,0.05475),	(8,0.04869),	(9,0.0433),	(10,0.03851)}	
    {\node[black] at \Point {\scalebox{0.3}{$\square$}};}
    
\end{tikzpicture}
\caption{$\rho=\lambda=0.8$}\label{fig:opt_d}
\end{subfigure}
\vspace{0.5cm}

\begin{tikzpicture}
    \begin{customlegend}
    [legend entries={ $k=0.5$,$k=1$,$k=2$,$k=10$},legend columns=-1,legend style={/tikz/every even column/.append style={column sep=0.8cm}}]   
    \addlegendimage{red,dashdotted,mark=star}
    \addlegendimage{green,smooth,mark=triangle}     
    \addlegendimage{blue,densely dotted,mark=*}    
    \addlegendimage{black,mark=square} 
    \end{customlegend}
\end{tikzpicture}
\caption{Approximate optimal service-rate sequence for varying values of $\rho$ and $k$.}
\label{fig:opt}
\end{figure}
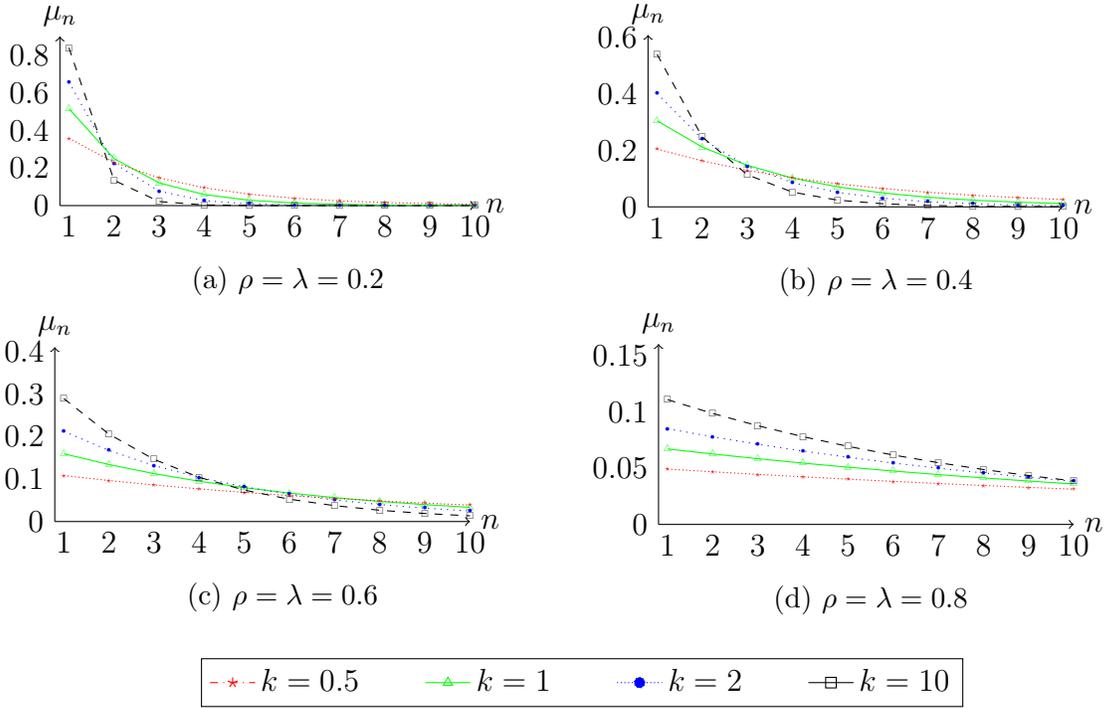

\begin{table}[H]
\centering
\footnotesize{
\begin{tabular}{|c|c|c|c|c|} \hline 
  & $\rho=0.2$  & $\rho=0.4$  & $\rho=0.6$ & $\rho=0.8$  \\ \hline 
 $k=0.5$ & $(5.18,0.04)$  & $(10.81,0.387)$ & $(25.72,8.23)$ & $(118.1,98.38)$  \\ \hline  
 $k=1$ & $(3.22,4.7^{-5})$  & $(6.86,0.014)$ & $(16.23,1.57)$ & $(69.78,48.48)$  \\ \hline 
  $k=2$ & $(2.23,1.6^{-7})$  & $(4.87,0.001)$ & $(11.78,0.36)$ & $(48.21,27.54)$  \\ \hline 
   $k=5$ & $(1.63,6.9^{-10})$  & $(3.66,3.3^{-4})$ & $(9.14,0.08)$ & $(36.19,16.29)$  \\ \hline 
    $k=10$ & $(1.44,8.2^{-12})$  & $(3.25,8.2^{-5})$ & $(8.25,0.04)$ & $(32.4,12.9)$  \\ \hline 
\end{tabular}
\caption{Approximate optimal expected delay and the optimal tail approximation: ($\E S$,$r_M$).}
\label{tbl:opt}}
\end{table}

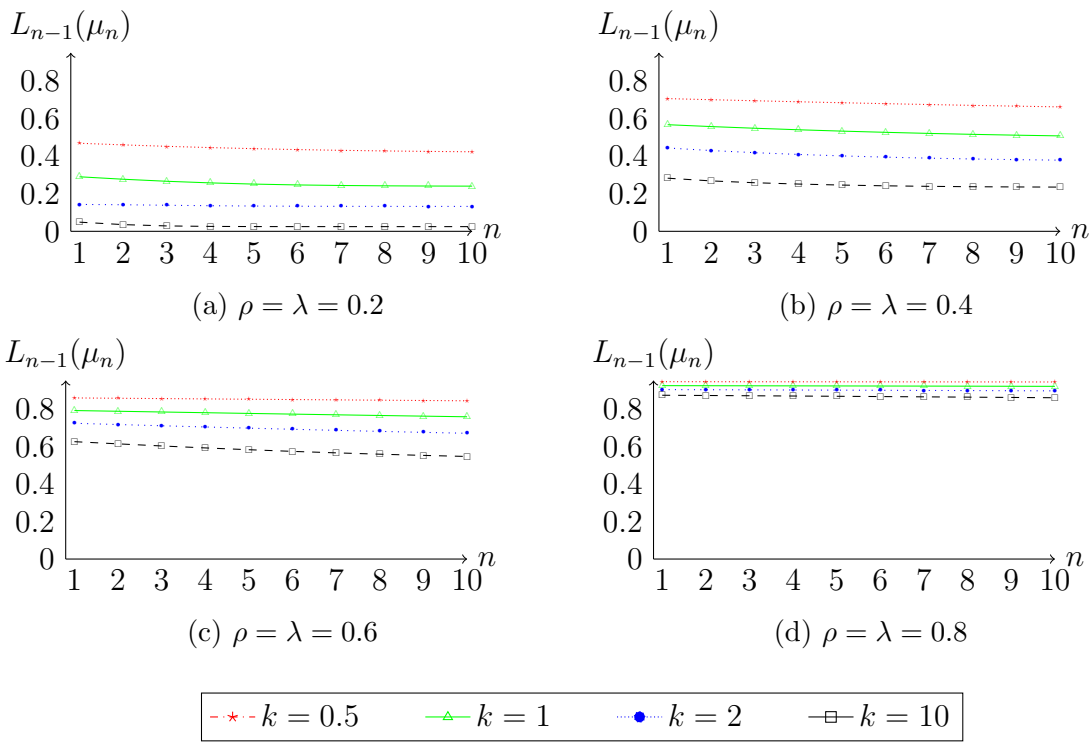
\begin{figure}[H]
\centering
\begin{subfigure}{.48\linewidth}
\begin{tikzpicture}[xscale=0.58,yscale=2.5]
  \def\xmin{0.8}
  \def\xmax{10}
  \def\ymin{0}
  \def\ymax{0.95}
    \draw[->] (\xmin,\ymin) -- (\xmax,\ymin) node[right] {$n$} ;
    \draw[->] (\xmin,\ymin) -- (\xmin,\ymax) node[above] {$L_{n-1}(\mu_n)$} ;
    \foreach \x in {1,2,3,4,5,6,7,8,9,10}
    \node at (\x,\ymin) [below] {\x};
    \foreach \y in {0,0.2,0.4,0.6,0.8}
    \node at (\xmin,\y) [left] {\y};
    
    \draw[smooth,green] (1,0.2902)--	(2,0.2772)--	(3,0.2663)--	(4,0.2577)--	(5,0.2513)--	(6,0.2468)--	(7,0.2439)--	(8,0.2421)--	(9,0.241)--	(10,0.2404);

    \foreach \Point in {(1,0.2902),	(2,0.2772),	(3,0.2663),	(4,0.2577),	(5,0.2513),	(6,0.2468),	(7,0.2439),	(8,0.2421),	(9,0.241),	(10,0.2404)}	
    {\node[green] at \Point {\scalebox{0.3}{$\triangle$}};}
    
    \draw[dotted,blue] (1,0.1431)--	(2,0.1414)--	(3,0.1387)--	(4,0.1362)--	(5,0.1345)--	(6,0.1335)--	(7,0.133)--	(8,0.1327)--	(9,0.1326)--	(10,0.1326);

    \foreach \Point in {(1,0.1431),	(2,0.1414),	(3,0.1387),	(4,0.1362),	(5,0.1345),	(6,0.1335),	(7,0.133),	(8,0.1327),	(9,0.1326),	(10,0.1326)}
    {\node[blue] at \Point {\scalebox{0.3}{\textbullet}};}
    
      \draw[densely dotted,red] ( 1 , 0.46863 )-- ( 2 , 0.45999 )-- ( 3 , 0.4523 )-- ( 4 , 0.44557 )-- ( 5 , 0.4398 )-- ( 6 , 0.43494 )-- ( 7 , 0.43095 )-- ( 8 , 0.42773 )-- ( 9 , 0.42521 )-- ( 10 , 0.42326 );

    \foreach \Point in {( 1 , 0.46863 ), ( 2 , 0.45999 ), ( 3 , 0.4523 ), ( 4 , 0.44557 ), ( 5 , 0.4398 ), ( 6 , 0.43494 ), ( 7 , 0.43095 ), ( 8 , 0.42773 ), ( 9 , 0.42521 ), ( 10 , 0.42326 )}	
    {\node[red] at \Point {\scalebox{0.3}{$\star$}};}
    
    \draw[dashed,black] (1,0.04869)--	(2,0.03571)--	(3,0.02868)--	(4,0.02609)--	(5,0.02529)--	(6,0.02505)--	(7,0.02498)--	(8,0.02496)--	(9,0.02496)--	(10,0.02495);

    \foreach \Point in {(1,0.04869),	(2,0.03571),	(3,0.02868),	(4,0.02609),	(5,0.02529),	(6,0.02505),	(7,0.02498),	(8,0.02496),	(9,0.02496),	(10,0.02495)}	
    {\node[black] at \Point {\scalebox{0.3}{$\square$}};}
\end{tikzpicture}
\caption{$\rho=\lambda=0.2$}\label{fig:lst_a}
\end{subfigure}
\begin{subfigure}{.48\linewidth}
\begin{tikzpicture}[xscale=0.58,yscale=2.5]
  \def\xmin{0.8}
  \def\xmax{10}
  \def\ymin{0}
  \def\ymax{0.95}
    \draw[->] (\xmin,\ymin) -- (\xmax,\ymin) node[right] {$n$} ;
    \draw[->] (\xmin,\ymin) -- (\xmin,\ymax) node[above] {$L_{n-1}(\mu_n)$} ;
    \foreach \x in {1,2,3,4,5,6,7,8,9,10}
    \node at (\x,\ymin) [below] {\x};
    \foreach \y in {0,0.2,0.4,0.6,0.8}
    \node at (\xmin,\y) [left] {\y};
    
    \draw[smooth,green] (1,0.5674)--	(2,0.5574)--	(3,0.5485)--	(4,0.5404)--	(5,0.5332)--	(6,0.5268)--	(7,0.521)--	(8,0.516)--	(9,0.5116)--	(10,0.5078);

    \foreach \Point in {(1,0.5674),	(2,0.5574),	(3,0.5485),	(4,0.5404),	(5,0.5332),	(6,0.5268),	(7,0.521),	(8,0.516),	(9,0.5116),	(10,0.5078)}	
    {\node[green] at \Point {\scalebox{0.3}{$\triangle$}};}
    
    \draw[dotted,blue] (1,0.4435)--	(2,0.4298)--	(3,0.4185)--	(4,0.4091)--	(5,0.4013)--	(6,0.3947)--	(7,0.3893)--	(8,0.3849)--	(9,0.3814)--	(10,0.3787);

    \foreach \Point in {(1,0.4435),	(2,0.4298),	(3,0.4185),	(4,0.4091),	(5,0.4013),	(6,0.3947),	(7,0.3893),	(8,0.3849),	(9,0.3814),	(10,0.3787)}
    {\node[blue] at \Point {\scalebox{0.3}{\textbullet}};}
    
      \draw[densely dotted,red] ( 1 , 0.70621 )-- ( 2 , 0.70039 )-- ( 3 , 0.69478 )-- ( 4 , 0.68939 )-- ( 5 , 0.68423 )-- ( 6 , 0.6793 )-- ( 7 , 0.67463 )-- ( 8 , 0.6702 )-- ( 9 , 0.66604 )-- ( 10 , 0.66213 );

    \foreach \Point in {( 1 , 0.70621 ), ( 2 , 0.70039 ), ( 3 , 0.69478 ), ( 4 , 0.68939 ), ( 5 , 0.68423 ), ( 6 , 0.6793 ), ( 7 , 0.67463 ), ( 8 , 0.6702 ), ( 9 , 0.66604 ), ( 10 , 0.66213 )}	
    {\node[red] at \Point {\scalebox{0.3}{$\star$}};}
    
    \draw[dashed,black](1,0.2827)--	(2,0.2682)--	(3,0.2585)--	(4,0.2515)--	(5,0.2461)--	(6,0.2421)--	(7,0.2393)--	(8,0.2373)--	(9,0.2361)--	(10,0.2353);

    \foreach \Point in {(1,0.2827),	(2,0.2682),	(3,0.2585),	(4,0.2515),	(5,0.2461),	(6,0.2421),	(7,0.2393),	(8,0.2373),	(9,0.2361),	(10,0.2353)}	
    {\node[black] at \Point {\scalebox{0.3}{$\square$}};}
\end{tikzpicture}
\caption{$\rho=\lambda=0.4$}\label{fig:lst_b}
\end{subfigure}

\begin{subfigure}{.48\linewidth}
\begin{tikzpicture}[xscale=0.58,yscale=2.5]
  \def\xmin{0.8}
  \def\xmax{10}
  \def\ymin{0}
  \def\ymax{0.95}
    \draw[->] (\xmin,\ymin) -- (\xmax,\ymin) node[right] {$n$} ;
    \draw[->] (\xmin,\ymin) -- (\xmin,\ymax) node[above] {$L_{n-1}(\mu_n)$} ;
    \foreach \x in {1,2,3,4,5,6,7,8,9,10}
    \node at (\x,\ymin) [below] {\x};
    \foreach \y in {0,0.2,0.4,0.6,0.8}
    \node at (\xmin,\y) [left] {\y};
    
   \draw[smooth,green] (1,0.7903)--	(2,0.7865)--	(3,0.7826)--	(4,0.7788)--	(5,0.775)--	(6,0.7713)--	(7,0.7676)--	(8,0.764)--	(9,0.7605)--	(10,0.7571);

    \foreach \Point in {(1,0.7903),	(2,0.7865),	(3,0.7826),	(4,0.7788),	(5,0.775),	(6,0.7713),	(7,0.7676),	(8,0.764),	(9,0.7605),	(10,0.7571)}	
    {\node[green] at \Point {\scalebox{0.3}{$\triangle$}};}
    
    \draw[dotted,blue] (1,0.7225)--	(2,0.7159)--	(3,0.7093)--	(4,0.7029)--	(5,0.6966)--	(6,0.6906)--	(7,0.6849)--	(8,0.6794)--	(9,0.6742)--	(10,0.6693);

    \foreach \Point in {(1,0.7225),	(2,0.7159),	(3,0.7093),	(4,0.7029),	(5,0.6966),	(6,0.6906),	(7,0.6849),	(8,0.6794),	(9,0.6742),	(10,0.6693)}
    {\node[blue] at \Point {\scalebox{0.3}{\textbullet}};}
    
      \draw[densely dotted,red] ( 1 , 0.85737 )-- ( 2 , 0.85566 )-- ( 3 , 0.85394 )-- ( 4 , 0.85221 )-- ( 5 , 0.85049 )-- ( 6 , 0.84878 )-- ( 7 , 0.84708 )-- ( 8 , 0.84539 )-- ( 9 , 0.84372 )-- ( 10 , 0.84207 );

    \foreach \Point in {( 1 , 0.85737 ), ( 2 , 0.85566 ), ( 3 , 0.85394 ), ( 4 , 0.85221 ), ( 5 , 0.85049 ), ( 6 , 0.84878 ), ( 7 , 0.84708 ), ( 8 , 0.84539 ), ( 9 , 0.84372 ), ( 10 , 0.84207 )}	
    {\node[red] at \Point {\scalebox{0.3}{$\star$}};}
    
    \draw[dashed,black] (1,0.6246)--	(2,0.613)--	(3,0.6018)--	(4,0.5914)--	(5,0.5818)--	(6,0.573)--	(7,0.5651)--	(8,0.558)--	(9,0.5516)--	(10,0.5459);

    \foreach \Point in {(1,0.6246),	(2,0.613),	(3,0.6018),	(4,0.5914),	(5,0.5818),	(6,0.573),	(7,0.5651),	(8,0.558),	(9,0.5516),	(10,0.5459)}	
    {\node[black] at \Point {\scalebox{0.3}{$\square$}};}
\end{tikzpicture}
\caption{$\rho=\lambda=0.6$}\label{fig:lst_c}
\end{subfigure}
\begin{subfigure}{.48\linewidth}
\begin{tikzpicture}[xscale=0.58,yscale=2.5]
  \def\xmin{0.8}
  \def\xmax{10}
  \def\ymin{0}
  \def\ymax{0.95}
    \draw[->] (\xmin,\ymin) -- (\xmax,\ymin) node[right] {$n$} ;
    \draw[->] (\xmin,\ymin) -- (\xmin,\ymax) node[above] {$L_{n-1}(\mu_n)$} ;
    \foreach \x in {1,2,3,4,5,6,7,8,9,10}
    \node at (\x,\ymin) [below] {\x};
    \foreach \y in {0,0.2,0.4,0.6,0.8}
    \node at (\xmin,\y) [left] {\y};
    
    \draw[smooth,green] (1,0.9228)--	(2,0.9224)--	(3,0.922)--	(4,0.9216)--	(5,0.9212)--	(6,0.9208)--	(7,0.9204)--	(8,0.92)--	(9,0.9195)--	(10,0.9191);

    \foreach \Point in {(1,0.9228),	(2,0.9224),	(3,0.922),	(4,0.9216),	(5,0.9212),	(6,0.9208),	(7,0.9204),	(8,0.92),	(9,0.9195),	(10,0.9191)}	
    {\node[green] at \Point {\scalebox{0.3}{$\triangle$}};}
    
    \draw[dotted,blue] (1,0.9018)--	(2,0.9011)--	(3,0.9004)--	(4,0.8996)--	(5,0.8989)--	(6,0.8981)--	(7,0.8974)--	(8,0.8966)--	(9,0.8958)--	(10,0.895);

    \foreach \Point in {(1,0.9018),	(2,0.9011),	(3,0.9004),	(4,0.8996),	(5,0.8989),	(6,0.8981),	(7,0.8974),	(8,0.8966),	(9,0.8958),	(10,0.895)}
    {\node[blue] at \Point {\scalebox{0.3}{\textbullet}};}
    
      \draw[densely dotted,red] ( 1 , 0.94383 )-- ( 2 , 0.94366 )-- ( 3 , 0.94349 )-- ( 4 , 0.94332 )-- ( 5 , 0.94315 )-- ( 6 , 0.94298 )-- ( 7 , 0.94282 )-- ( 8 , 0.94265 )-- ( 9 , 0.94249 )-- ( 10 , 0.94232 );

    \foreach \Point in {( 1 , 0.94383 ), ( 2 , 0.94366 ), ( 3 , 0.94349 ), ( 4 , 0.94332 ), ( 5 , 0.94315 ), ( 6 , 0.94298 ), ( 7 , 0.94282 ), ( 8 , 0.94265 ), ( 9 , 0.94249 ), ( 10 , 0.94232 )}	
    {\node[red] at \Point {\scalebox{0.3}{$\star$}};}
    
    \draw[dashed,black] (1,0.8717)--	(2,0.8704)--	(3,0.869)--	(4,0.8676)--	(5,0.8662)--	(6,0.8647)--	(7,0.8632)--	(8,0.8616)--	(9,0.86)--	(10,0.8584);

    \foreach \Point in {(1,0.8717),	(2,0.8704),	(3,0.869),	(4,0.8676),	(5,0.8662),	(6,0.8647),	(7,0.8632),	(8,0.8616),	(9,0.86),	(10,0.8584)}	
    {\node[black] at \Point {\scalebox{0.3}{$\square$}};}
\end{tikzpicture}
\caption{$\rho=\lambda=0.8$}\label{fig:lst_d}
\end{subfigure}
\vspace{0.5cm}

\begin{tikzpicture}
    \begin{customlegend}
    [legend entries={$k=0.5$, $k=1$,$k=2$,$k=10$},legend columns=-1,legend style={/tikz/every even column/.append style={column sep=0.8cm}}]   
	\addlegendimage{red,dashdotted,mark=star}     
    \addlegendimage{green,smooth,mark=triangle}     
    \addlegendimage{blue,densely dotted,mark=*}   
    \addlegendimage{black,mark=square} 
    \end{customlegend}
\end{tikzpicture}
\caption{Laplace transform sequence ($L_{n-1}(\mu_n)$) corresponding to the approximate optimal service-rate sequence for varying values of $\rho$ and $k$.}
\label{fig:lst}
\end{figure}

In the special case of Poisson arrivals ($k=1$) it is interesting to observe that $\mu_1\eqsim 1-\sqrt{\rho}$, as seen in Figure \ref{fig:mu1_rho}. Thus, a reasonable rough approximation for the optimal service-rate sequence given by
\[
\mu_n=\left\lbrace\begin{array}{cc}
1-\sqrt{\rho}, & n=1 \\
\sqrt{\rho}\mu_{n-1}, & n\geq 2.
\end{array}\right.
\]
The value $\sqrt{\rho}$ appeared in two places before: (1) The solution to the equation $L_0(\alpha)=L_1(\alpha(1-\alpha))$ is $\alpha=1-\sqrt{\rho}$, as was elaborated in Observation 2 of Section \ref{sec:geometric} (see also Figure \ref{fig:l_alpha}). This seems to be a critical point for the limit function $\ell(\alpha)$ for geometric service-rate sequence (with rate $\alpha$). (2) The optimal tail decay rate given in Proposition \ref{prop:TAP} is $\sqrt{\ell}$, where in the Poisson case we observe that $\ell\eqsim C\rho$ where $C$ is a constant that was in the range of $(1,1.15)$ in all examples computed.

\begin{figure}[H]
\centering
\begin{tikzpicture}[xscale=7,yscale=3.5]
  \def\xmin{0}
  \def\xmax{1}
  \def\ymin{0}
  \def\ymax{1.05}
    \draw[->] (\xmin,\ymin) -- (\xmax,\ymin) node[right] {$\rho$} ;
    \draw[->] (\xmin,\ymin) -- (\xmin,\ymax); 
    \foreach \x in {0.2,0.4,0.6,0.8,1}
    \node at (\x,\ymin) [below] {\x};
    \foreach \y in {0,0.2,0.4,0.6,0.8,1}
    \node at (\xmin,\y) [left] {\y};
	
    \draw[red,densely dashed,domain=0:1] plot (\x, {1-sqrt(\x)}) ;
    
    \draw[blue] (0.001,0.9) -- (0.05,0.76211)--	(0.1,0.65953)--	(0.15,0.57994)--	(0.2,0.50975)--	(0.25,0.45226)--	(0.3,0.39805)--	(0.35,0.3446)--	(0.4,0.30389)--	(0.45,0.26407)--	(0.5,0.22542)--	(0.55,0.19067)--	(0.6,0.15887)--	(0.65,0.13086)--	(0.7,0.1069)--	(0.75,0.08578)--	(0.8,0.06824)--	(0.85,0.05196)--	(0.9,0.03623)--	(0.95,0.01991);

    \node[draw=none,color=red] at (0.8,0.23) {$1-\sqrt{\rho}$};
    \node[draw=none,color=blue] at (0.6,0.1) {$\mu_1$};
    
\end{tikzpicture}
 \caption{Approximation of optimal $\mu_1$ as a function of $\rho$ for the Poisson arrival case ($k=1$).}\label{fig:mu1_rho}
\end{figure}
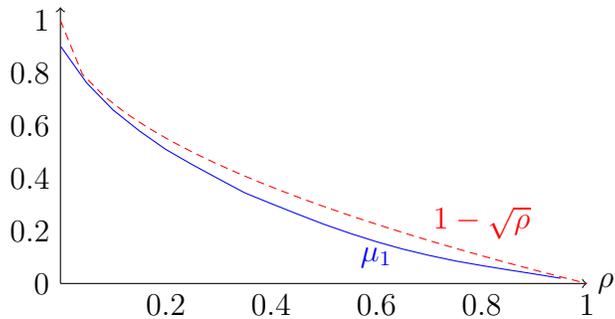

Let $\rho_0=\rho=\frac{\lambda}{\mu}$ and let $\rho_n:=\frac{\lambda p_{n}}{\mu-\sum_{i=1}^{n}\mu_i}$ denote the effective utilization of the sub-system excluding the first $n$ servers for $n\geq 1$. In Figure \ref{fig:rho} we see that the effective utilization sequence, given the approximate optimal service-rate sequence, is decreasing with $n$ for all parameter values. An interesting numerical result is that in all examples the tail of the utilization level sequence decays geometrically. The rate of decay is slightly higher than $1-\mu_1$, that is to say the effective utilization decreases at a slower rate than the service-rate sequence, as expected.

\begin{figure}[H]
\centering
\begin{subfigure}{.48\linewidth}
\begin{tikzpicture}[xscale=0.6,yscale=2.5]
  \def\xmin{0.8}
  \def\xmax{10}
  \def\ymin{0}
  \def\ymax{0.9}
    \draw[->] (\xmin,\ymin) -- (\xmax,\ymin) node[right] {$n$} ;
    \draw[->] (\xmin,\ymin) -- (\xmin,\ymax) node[above] {$\rho_n$} ;
    \foreach \x in {1,2,3,4,5,6,7,8,9,10}
    \node at (\x,\ymin) [below] {\x};
    \foreach \y in {0,0.2,0.4,0.6,0.8}
    \node at (\xmin,\y) [left] {\y};
    
    \draw[smooth,green] (1,0.2)--	(2,0.1183860622)--	(3,0.066929141)--	(4,0.0363566839)--	(5,0.0191111039)--	(6,0.0097963413)--	(7,0.0049325924)--	(8,0.0024542598)--	(9,0.0012120319)--	(10,0.000595864);

    \foreach \Point in {(1,0.2),	(2,0.1183860622),	(3,0.066929141),	(4,0.0363566839),	(5,0.0191111039),	(6,0.0097963413),	(7,0.0049325924),	(8,0.0024542598),	(9,0.0012120319),	(10,0.000595864)}	
    {\node[green] at \Point {\scalebox{0.3}{$\triangle$}};}
    
    \draw[dotted,blue] (1,0.2)--	(2,0.083551161774)--	(3,0.034490599154)--	(4,0.01396840367)--	(5,0.005555948128)--	(6,0.002182029091)--	(7,0.00085062238)--	(8,0.000330321406)--	(9,0.000128035905)--	(10,0.000049586148);

    \foreach \Point in {(1,0.2),	(2,0.083551161774),	(3,0.034490599154),	(4,0.01396840367),	(5,0.005555948128),	(6,0.002182029091),	(7,0.00085062238),	(8,0.000330321406),	(9,0.000128035905),	(10,0.000049586148)}
    {\node[blue] at \Point {\scalebox{0.3}{\textbullet}};}
    
      \draw[densely dotted,red] ( 1 , 0.2 )-- ( 2 , 0.14605 )-- ( 3 , 0.10468 )-- ( 4 , 0.07378 )-- ( 5 , 0.05123 )-- ( 6 , 0.03511 )-- ( 7 , 0.02379 )-- ( 8 , 0.01598 )-- ( 9 , 0.01065 )-- ( 10 , 0.00706 );

    \foreach \Point in {( 1 , 0.2 ), ( 2 , 0.14605 ), ( 3 , 0.10468 ), ( 4 , 0.07378 ), ( 5 , 0.05123 ), ( 6 , 0.03511 ), ( 7 , 0.02379 ), ( 8 , 0.01598 ), ( 9 , 0.01065 ), ( 10 , 0.00706 )}	
    {\node[red] at \Point {\scalebox{0.3}{$\star$}};}
    
    \draw[dashed,black] (1,0.2)--	(2,0.060113988117692)--	(3,0.0132504723393969)--	(4,0.00234555499039934)--	(5,0.000377668558325819)--	(6,0.00000589468822303439)--	(7,0.00000091149741375569)--	(8,0.00000014055816846642)--	(9,0.00000002165737996784)--	(10,0.000000033362039986);

    \foreach \Point in {(1,0.2),	(2,0.060113988117692),	(3,0.0132504723393969),	(4,0.00234555499039934),	(5,0.000377668558325819),	(6,0.00000589468822303439),	(7,0.00000091149741375569),	(8,0.00000014055816846642),	(9,0.00000002165737996784),	(10,0.000000033362039986)}	
    {\node[black] at \Point {\scalebox{0.3}{$\square$}};}
\end{tikzpicture}
\caption{$\rho=\lambda=0.2$}\label{fig:rho_a}
\end{subfigure}
\begin{subfigure}{.48\linewidth}
\begin{tikzpicture}[xscale=0.6,yscale=2.5]
  \def\xmin{0.8}
  \def\xmax{10}
  \def\ymin{0}
  \def\ymax{0.9}
    \draw[->] (\xmin,\ymin) -- (\xmax,\ymin) node[right] {$n$} ;
    \draw[->] (\xmin,\ymin) -- (\xmin,\ymax) node[above] {$\rho_n$} ;
    \foreach \x in {1,2,3,4,5,6,7,8,9,10}
    \node at (\x,\ymin) [below] {\x};
    \foreach \y in {0,0.2,0.4,0.6,0.8}
    \node at (\xmin,\y) [left] {\y};
    
     \draw[smooth,green](1,0.4)--	(2,0.326547)--	(3,0.26189)--	(4,0.206649)--	(5,0.160665)--	(6,0.123246)--	(7,0.093401)--	(8,0.070016)--	(9,0.051978)--	(10,0.038258);

    \foreach \Point in {(1,0.4),	(2,0.326547),	(3,0.26189),	(4,0.206649),	(5,0.160665),	(6,0.123246),	(7,0.093401),	(8,0.070016),	(9,0.051978),	(10,0.038258)}	
    {\node[green] at \Point {\scalebox{0.3}{$\triangle$}};}
    
    \draw[dotted,blue] (1,0.4)--	(2,0.296834)--	(3,0.2134753)--	(4,0.1494922)--	(5,0.1023365)--	(6,0.0687102)--	(7,0.0453788)--	(8,0.0295585)--	(9,0.0190363)--	(10,0.0121489);

    \foreach \Point in {(1,0.4),	(2,0.296834),	(3,0.2134753),	(4,0.1494922),	(5,0.1023365),	(6,0.0687102),	(7,0.0453788),	(8,0.0295585),	(9,0.0190363),	(10,0.0121489)}
    {\node[blue] at \Point {\scalebox{0.3}{\textbullet}};}
    
      \draw[densely dotted,red] ( 1 , 0.4 )-- ( 2 , 0.35539 )-- ( 3 , 0.31316 )-- ( 4 , 0.27373 )-- ( 5 , 0.23741 )-- ( 6 , 0.20437 )-- ( 7 , 0.17466 )-- ( 8 , 0.14824 )-- ( 9 , 0.125 )-- ( 10 , 0.10474 );

    \foreach \Point in {( 1 , 0.4 ), ( 2 , 0.35539 ), ( 3 , 0.31316 ), ( 4 , 0.27373 ), ( 5 , 0.23741 ), ( 6 , 0.20437 ), ( 7 , 0.17466 ), ( 8 , 0.14824 ), ( 9 , 0.125 ), ( 10 , 0.10474 )}	
    {\node[red] at \Point {\scalebox{0.3}{$\star$}};}
    
    \draw[dashed,black] (1,0.4)--	(2,0.246085861)--	(3,0.143626829)--	(4,0.080812862)--	(5,0.044233484)--	(6,0.023696777)--	(7,0.012487959)--	(8,0.006502927)--	(9,0.003358763)--	(10,0.001725664);

    \foreach \Point in {(1,0.4),	(2,0.246085861),	(3,0.143626829),	(4,0.080812862),	(5,0.044233484),	(6,0.023696777),	(7,0.012487959),	(8,0.006502927),	(9,0.003358763),	(10,0.001725664)}	
    {\node[black] at \Point {\scalebox{0.3}{$\square$}};}
\end{tikzpicture}
\caption{$\rho=\lambda=0.4$}\label{fig:rho_b}
\end{subfigure}

\begin{subfigure}{.48\linewidth}
\begin{tikzpicture}[xscale=0.6,yscale=2.5]
  \def\xmin{0.8}
  \def\xmax{10}
  \def\ymin{0}
  \def\ymax{0.9}
    \draw[->] (\xmin,\ymin) -- (\xmax,\ymin) node[right] {$n$} ;
    \draw[->] (\xmin,\ymin) -- (\xmin,\ymax) node[above] {$\rho_n$} ;
    \foreach \x in {1,2,3,4,5,6,7,8,9,10}
    \node at (\x,\ymin) [below] {\x};
    \foreach \y in {0,0.2,0.4,0.6,0.8}
    \node at (\xmin,\y) [left] {\y};
    
    \draw[smooth,green] (1,0.6)--	(2,0.56397)--	(3,0.52753)--	(4,0.49104)--	(5,0.45485)--	(6,0.41927)--	(7,0.38461)--	(8,0.35113)--	(9,0.31907)--	(10,0.28861);

    \foreach \Point in {(1,0.6),	(2,0.56397),	(3,0.52753),	(4,0.49104),	(5,0.45485),	(6,0.41927),	(7,0.38461),	(8,0.35113),	(9,0.31907),	(10,0.28861)}	
    {\node[green] at \Point {\scalebox{0.3}{$\triangle$}};}
    
    \draw[dotted,blue] (1,0.6)--	(2,0.54996)--	(3,0.49948)--	(4,0.44947)--	(5,0.4008)--	(6,0.35423)--	(7,0.31037)--	(8,0.26967)--	(9,0.23244)--	(10,0.19881);

    \foreach \Point in {(1,0.6),	(2,0.54996),	(3,0.49948),	(4,0.44947),	(5,0.4008),	(6,0.35423),	(7,0.31037),	(8,0.26967),	(9,0.23244),	(10,0.19881)}
    {\node[blue] at \Point {\scalebox{0.3}{\textbullet}};}
    
      \draw[densely dotted,red] ( 1 , 0.6 )-- ( 2 , 0.57678 )-- ( 3 , 0.55335 )-- ( 4 , 0.52981 )-- ( 5 , 0.50624 )-- ( 6 , 0.48275 )-- ( 7 , 0.45941 )-- ( 8 , 0.43633 )-- ( 9 , 0.41359 )-- ( 10 , 0.39125 );

    \foreach \Point in {( 1 , 0.6 ), ( 2 , 0.57678 ), ( 3 , 0.55335 ), ( 4 , 0.52981 ), ( 5 , 0.50624 ), ( 6 , 0.48275 ), ( 7 , 0.45941 ), ( 8 , 0.43633 ), ( 9 , 0.41359 ), ( 10 , 0.39125 )}	
    {\node[red] at \Point {\scalebox{0.3}{$\star$}};}
    
    \draw[dashed,black] (1,0.6)--	(2,0.527194)--	(3,0.454568)--	(4,0.384824)--	(5,0.320138)--	(6,0.261998)--	(7,0.211188)--	(8,0.167877)--	(9,0.131766)--	(10,0.102239);

    \foreach \Point in {(1,0.6),	(2,0.527194),	(3,0.454568),	(4,0.384824),	(5,0.320138),	(6,0.261998),	(7,0.211188),	(8,0.167877),	(9,0.131766),	(10,0.102239)}	
    {\node[black] at \Point {\scalebox{0.3}{$\square$}};}
\end{tikzpicture}
\caption{$\rho=\lambda=0.6$}\label{fig:rho_c}
\end{subfigure}
\begin{subfigure}{.48\linewidth}
\begin{tikzpicture}[xscale=0.6,yscale=2.5]
  \def\xmin{0.8}
  \def\xmax{10}
  \def\ymin{0}
  \def\ymax{0.9}
    \draw[->] (\xmin,\ymin) -- (\xmax,\ymin) node[right] {$n$} ;
    \draw[->] (\xmin,\ymin) -- (\xmin,\ymax) node[above] {$\rho_n$} ;
    \foreach \x in {1,2,3,4,5,6,7,8,9,10}
    \node at (\x,\ymin) [below] {\x};
    \foreach \y in {0,0.2,0.4,0.6,0.8}
    \node at (\xmin,\y) [left] {\y};
    
    \draw[smooth,green] (1,0.8)--	(2,0.7912)--	(3,0.7821)--	(4,0.7729)--	(5,0.7634)--	(6,0.7537)--	(7,0.7438)--	(8,0.7336)--	(9,0.7233)--	(10,0.7129);

    \foreach \Point in {(1,0.8),	(2,0.7912),	(3,0.7821),	(4,0.7729),	(5,0.7634),	(6,0.7537),	(7,0.7438),	(8,0.7336),	(9,0.7233),	(10,0.7129)}	
    {\node[green] at \Point {\scalebox{0.3}{$\triangle$}};}
    
    \draw[dotted,blue] (1,0.8)--	(2,0.7887)--	(3,0.777)--	(4,0.7649)--	(5,0.7523)--	(6,0.7394)--	(7,0.726)--	(8,0.7123)--	(9,0.6982)--	(10,0.6838);

    \foreach \Point in {(1,0.8),	(2,0.7887),	(3,0.777),	(4,0.7649),	(5,0.7523),	(6,0.7394),	(7,0.726),	(8,0.7123),	(9,0.6982),	(10,0.6838)}
    {\node[blue] at \Point {\scalebox{0.3}{\textbullet}};}
    
      \draw[densely dotted,red] ( 1 , 0.8 )-- ( 2 , 0.79403 )-- ( 3 , 0.78797 )-- ( 4 , 0.78181 )-- ( 5 , 0.77556 )-- ( 6 , 0.76922 )-- ( 7 , 0.7628 )-- ( 8 , 0.75629 )-- ( 9 , 0.74971 )-- ( 10 , 0.74306 );

    \foreach \Point in {( 1 , 0.8 ), ( 2 , 0.79403 ), ( 3 , 0.78797 ), ( 4 , 0.78181 ), ( 5 , 0.77556 ), ( 6 , 0.76922 ), ( 7 , 0.7628 ), ( 8 , 0.75629 ), ( 9 , 0.74971 ), ( 10 , 0.74306 )}	
    {\node[red] at \Point {\scalebox{0.3}{$\star$}};}
    
    \draw[dashed,black] (1,0.8)--	(2,0.7843)--	(3,0.7677)--	(4,0.7503)--	(5,0.7321)--	(6,0.7132)--	(7,0.6936)--	(8,0.6734)--	(9,0.6525)--	(10,0.6311);

    \foreach \Point in {(1,0.8),	(2,0.7843),	(3,0.7677),	(4,0.7503),	(5,0.7321),	(6,0.7132),	(7,0.6936),	(8,0.6734),	(9,0.6525),	(10,0.6311)}	
    {\node[black] at \Point {\scalebox{0.3}{$\square$}};}
\end{tikzpicture}
\caption{$\rho=\lambda=0.8$}\label{fig:rho_d}
\end{subfigure}
\vspace{0.5cm}

\begin{tikzpicture}
    \begin{customlegend}
    [legend entries={$k=0.5$, $k=1$,$k=2$,$k=10$},legend columns=-1,legend style={/tikz/every even column/.append style={column sep=0.8cm}}]   
	\addlegendimage{red,dashdotted,mark=star}     
    \addlegendimage{green,smooth,mark=triangle}     
    \addlegendimage{blue,densely dotted,mark=*}   
    \addlegendimage{black,mark=square} 
    \end{customlegend}
\end{tikzpicture}
\caption{Effective utilization rate sequence $\rho_n$ corresponding to the approximate optimal service-rate sequence for varying values of $\rho$ and $k$.}
\label{fig:rho}
\end{figure}
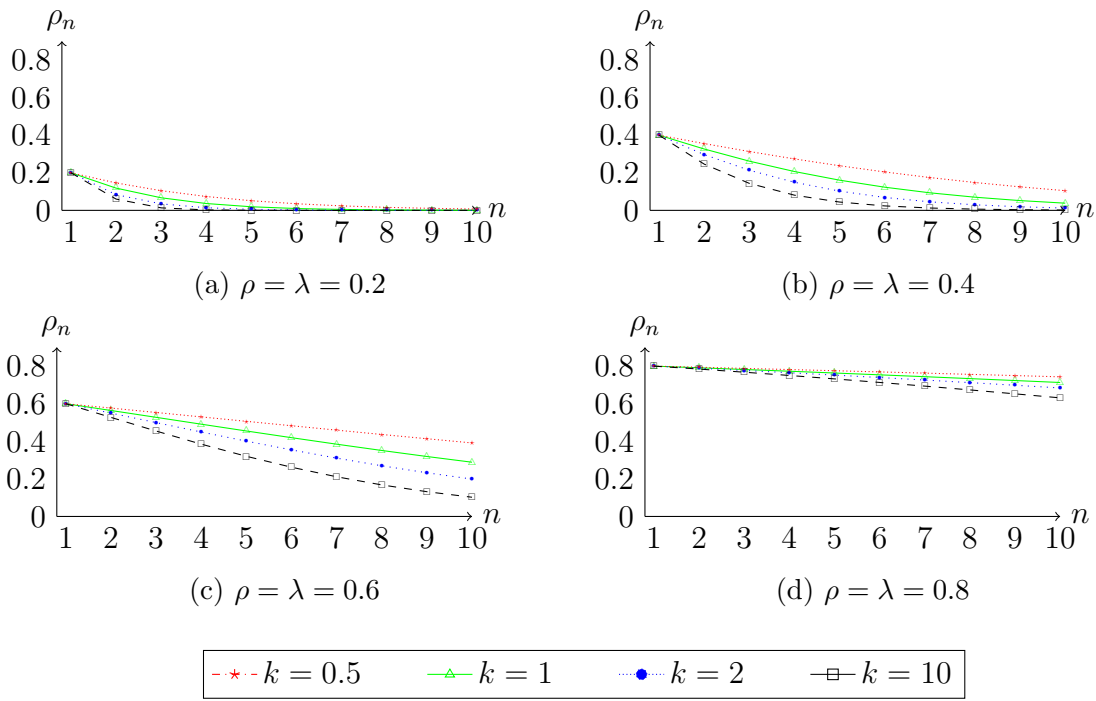

\subsection*{Summary of numerical results}

The approximat optimal service-rate sequence is close to geometric with decay rate $1-\mu_1$. In the special case of a Poisson arrival process we observed that $\mu_1\eqsim 1-\sqrt{\rho}$ and $\ell$ was slightly larger than $\rho$. When considering a fixed $\rho$, lower variance of the exogenous arrival stream leads to a higher service-rate for the first server, along with a faster decline to zero. However, as $\rho$ increases the service capacity is allocated more ``uniformly'' (with a lower decay rate). The effective utilization sequence of sub-systems, $\rho_n$, under the approximate optimal service-rate sequence is decreasing. Furthermore, this sequence has a geometric tail with a slower decay rate than the optimal service-rate sequence.

Recall that in order for the approximation to be reasonable we would like the sequence of overflow LST $L_{n-1}(\mu_n)$ to approach a constant at a quick rate. Indeed, we observe that it stabilizes very fast, suggesting that our approximation scheme using its limit is accurate even for a small number of DP steps $M$. This may provide some explanation as to why the optimal sequence seems geometric from the start. If the sequence $L_{n-1}(\mu_n)$ is more or less constant from the start then the solution to the TAP from Proposition \ref{prop:TAP} is close to optimal. It would be interesting to find an analytical explanation for why the optimal service-rate sequence comes with a stable overflow LST sequence.

\section{Applications}\label{sec:applications} 
The analysis presented here can be modified in order to solve several other system design problems, for example:
\begin{enumerate}[label=\alph*.]
\item Multi-objective optimization: suppose that the system administrator can choose the number of servers as well as the capacity allocation. If $n$ servers are used then the system has a customer loss probability of $p_n$. The administrator may seek an optimal allocation with a constraint on the loss probability, e.g. $p_n\leq p<1$. The other way around is also an option: minimize $p_n$ subject to a constraint on the delay, $\E S\leq w$.
\item Suppose that the system administrator wants to maximize the number of users that wait less than some $\tau>0$ time threshold. This could be an exogenous performance measure or the case if customers do not pay if their delay is too long. The objective is now
\[
\min_{\{\mu_n\}_{n=1}^\infty}\sum_{n=1}^\infty q_n e^{-\mu_n\tau}.
\]
\item Customers are heterogeneous with respect to the utility from the speed of service, and balk from the system if the fastest available server is slower than their value. Assume that customer values are distributed according to a continuous distribution with a convex cdf $\Lambda$ such that $\Lambda(0)>0$. If the system wants to minimize the blocking probability then the objective becomes
\[
\min_{\{\mu_n\}_{n=1}^\infty}\sum_{n=1}^\infty q_n(1-\Lambda(\mu_n)).
\]
In this case the overflow distribution also requires a modification to $\tilde{L}_{n-1}(\mu_n)=L_{n-1}(\mu_n)\Lambda(\mu_n)$, in order to take into account the balking customers.
\item In general, our analysis can be applied to any convex function of the service rate.
\end{enumerate}

\section{Discussion}\label{sec:discussion} 

This paper analyses stability and expected delay in an infinite-server system with finite service capacity. In particular, the expected service-time minimizing allocation of service-rates is examined. It is shown that the optimal service-rate sequence is geometrically decreasing at the tail. Numerical analysis suggests that the optimal allocation is very close to geometric throughout the sequence, and not just at the tail. This property is related to the product form of the blocking probabilities from finite sub-systems. An interesting numerical observation is that in the Poisson case we have that $\ell\eqsim\rho$ (the limiting term in the blocking probability product) and consequentially the optimal tail decay is simply $\sqrt{\rho}$. An open challenge is to find analytical justification for this phenomenon.

The most important conclusion of the paper is that allocating capacity to heterogeneous servers under capacity constraints should be done with caution. Even if there is seemingly enough capacity for the incoming arrival rate, allocating too much capacity to the fast servers may lead to very long expected delay. In this paper we analysed an infinite server system but this conclusion is also relevant for finite server loss systems with very low blocking probability, in which case expected delay would be finite but potentially very big. Although this is most relevant for very large systems, the geometric tail behaviour implies that with a ``bad" allocation the expected delay can increase very fast with the number of servers and therefore the conclusion is still relevant for moderately sized systems, i.e.\ not necessarily hundreds of servers.

Stability analysis in the probabilistic sense of the underlying Markov chain $\mathbf{X}(t)\in\mathcal{S}$ is also of interest. Specifically, establishing necessary and sufficient conditions for the process to be positive recurrent. This can perhaps be achieved by considering the embedded Markov chain at arrival times. For the Markovian M/M/$\infty$ ordered system, lower-bound and upper-bound systems with simpler dynamics can possibly be constructed and analysed using matrix-analytic methods. Rigorous characterisation of the stability conditions could potentially also shed some light on the open problems discussed in the end of Section \ref{sec:stability}. In particular, establishing the exact necessary and sufficient conditions for expected finite delay. The distinction between positive and null recurrence is potentially the additional refinement required for dealing with the case of a service sequence on the boundary of the feasibility region.

It would be interesting to study the asymptotic optimal control problem of this system using heavy-traffic analysis, i.e.\ scaling the parameters by an appropriate rate function $r(n)$,
\[
r(n)\lambda^{(n)}\xrightarrow{n\to\infty}\lambda=\mu\xleftarrow{n\to\infty}r(n)\mu^{(n)},\quad \lambda^{(n)}<\mu^{(n)} \ \forall n\geq 1.
\]
A question that arises is whether there is an asymptotic analog to the optimal geometric sequence we have presented here, perhaps even in closed form. For $\rho<1$ there is always a feasible service-rate sequence, however, our approximations are less accurate as $\rho\uparrow 1$, so heavy-traffic approximations may yield better insight to such systems. Detailed heavy-traffic analysis for the homogeneous ranked M/M/$\infty$ system can be found in \cite{book_N1984}, and approximation analysis of blocking probabilities in multi-server systems can be found in \cite{W1984}. Approximations of the number of idle servers in many-server systems, such as \cite{EG2018}, may also be useful. Approximating our model requires the analysis of systems with non-homogeneous servers (see \cite{AGS2013}).

\section{Acknowledgements}\label{sec:acknowledge} 
The authors wish to thank Abhishek, Brian Fralix and Johan van Leeuwaarden for helpful discussions and comments. We are also grateful to an associate editor and two referees for their detailed comments and suggestions which greatly improved this paper. This research was supported by the ISRAEL SCIENCE FOUNDATION (grant No. 355/15).

\footnotesize{\bibliography{/Users/lironravner/Dropbox/University/Research/Full_Bibliography/BigBib}}

\end{document}